\title{Explicit Constructions for Poncelet Polygons
}
\author{ \it Leah Wrenn Berman, Gábor Gévay, \\ \it Jürgen Richter-Gebert, Serge Tabachnikov}

\documentclass{article}
\usepackage{graphicx}
\usepackage{amsthm}
\newtheorem*{theorem*}{Theorem}
\newtheorem*{theoremA*}{Theorem A}
\newtheorem*{theoremB*}{Theorem B}
\newtheorem*{theoremC*}{Theorem C}

\newtheorem{construction}{Construction}
\newtheorem{theorem}{Theorem}
\newtheorem{remark}{Remark}

\usepackage{amssymb}
\usepackage[font=small,labelfont=bf]{caption} 
\usepackage{color}
\usepackage{mathtools}
\usepackage{amsmath}
\usepackage[dvipsnames]{xcolor}
\usepackage[all]{nowidow}
\usepackage{float}

\usepackage{soul} 
\usepackage[normalem]{ulem} 

\usepackage[color links=true, backref=page]{hyperref} 

\hypersetup{
  colorlinks,
  citecolor=blue,
  linkcolor=Red,
  urlcolor=Blue
  }

\date{\vspace{-2ex}}

\begin{document}
\newpage
\maketitle





\def\upto{, \ldots ,}
\def\lines{\square}
\def\points{\bigcirc}
\def\lines{\vee}
\def\points{\wedge}
\parskip=1mm
\setlength{\parindent}{.5cm}

\begin{abstract}
We study the geometric structure of  Poncelet $n$-gons from a projective point of view. In particular we present explicit constructions of Poncelet $n$-gons for certain $n$ and derive algebraic characterisations in terms of bracket polynomials. Via the derivations in  \cite{ BGRGT24a} that connect
Poncelet polygons to movable $(N_4)$-configurations, the results of this article can be used to construct a large class of specific movable $(N_4)$-configurations, the \emph{trivial celestial 4-configurations}, which up to this point were all thought to be rigid and to require regular polygons for their construction.

\end{abstract}

\tableofcontents

\parskip=2mm

\section{Introduction}
Poncelet's Porism is one of the oldest Theorems in projective geometry, and it has deep connections to to several areas of mathematics such as elliptic functions, dynamical systems, billiard reflections and many more. 
In \cite{ BGRGT24a} we showed that there is also a relation
to the theory of $(N_4)$-configurations. These are configurations
with $n$ points and $n$ lines having 4 points on each line and 4 lines through each point.
 Poncelet polygons
can be used to create large classes of movable $(N_4)$-configurations. In a sense, these configurations inherit their flexibility from the underlying Poncelet polygon.

\medskip
In \cite{ BGRGT24a} the main relation between the two concepts was established.
In the present article we study the question of geometric constructions and algebraic characterisations of Poncelet polygons. This generates explicit geometric constructions for certain $(N_4)$ configurations. We cannot expect too much since the creation of Poncelet polygons may lead to the necessity of solving algebraic equations of potentially high degree.

\noindent
In what follows we will present geometric constrictions for
\begin{itemize}
\item a Poncelet 6-gon (Section 4.1)
\item a Poncelet 7-gon (Section 4.2)
\item a Poncelet 8-gon (Section 4.3)
\item a Poncelet $2n$-gon from a Poncelet $n$-gon  (Section 4.4)
\item a Poncelet 9-gon  (Section 4.5)
\item a Poncelet $\infty$-gon (Section 4.6)
\end{itemize}
Poncelet 5-gons can be trivially constructed, because five points as well as five tangents uniquely determine a conic.  Hence every five distinct points on a conic form a Poncelet 5-gon.

Furthermore we will give algebraic characterizations in terms of 
bracket polynomials (that underline the projectively invariant character of the topic) for
\begin{itemize}
\item Poncelet chains (Section 3.1)
\item Poncelet 7-gons (Section 3.2)
\item Poncelet 8-gons (Section 4.2)
\item Poncelet 9-gons  (Section 4.4)
\end{itemize}

The smallest Poncelet polygon for which we do not have an explicit geometric construction is a Poncelet 11-gon; determining this construction would require solving an equation of degree 5.  A collection of interactive animations related to this article can be found at:  \href{https://mathvisuals.org/Poncelet/}{\tt https://mathvisuals.org/Poncelet/}.

\section{Poncelet's Porism and $(N_4)$-configurations}

Poncelets Porism is a theorem about a chain of points and lines that simultaneously circumscribe one conic and inscribe another:

\noindent
{\bf Poncelet's Porism:\ }{\it
Let $\mathcal{A}$ and $\mathcal{B}$ be two conics in the projective plane.
If $(p_1\upto p_n)$ is a polygon whose vertices lie on $\mathcal{A}$  and whose edges are tangent to $\mathcal{B}$, then there exists such a  polygon starting with an arbitrary point on $\mathcal{A}$.
}

\begin{figure}[H]
\begin{center}
\includegraphics[width=0.9\textwidth]{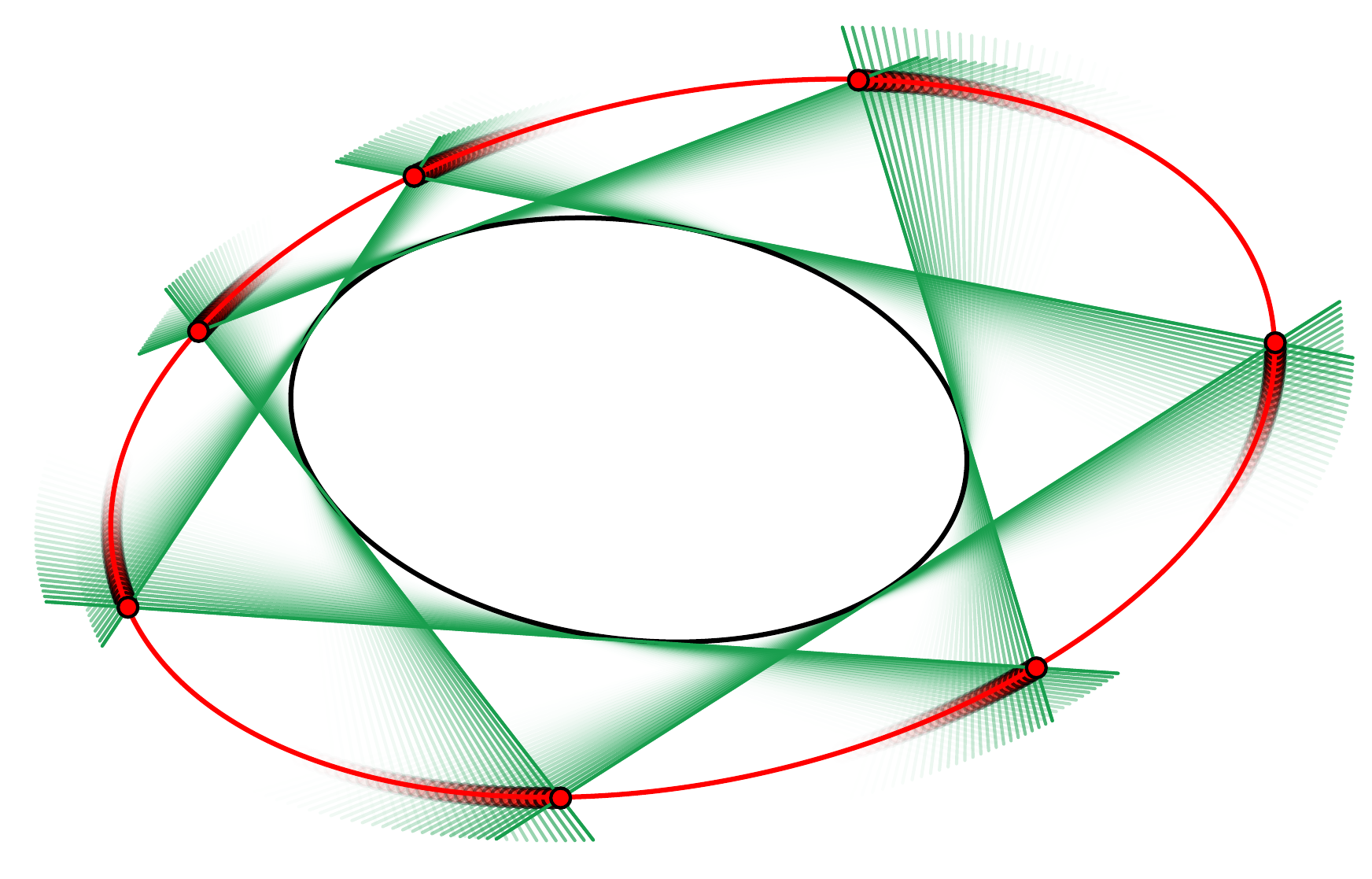}\
\begin{picture}(0,0)
\end{picture}
\end{center}
\captionof{figure}{Poncelet's Theorem.}\label{fig:Mov1}
\end{figure}

If the two conics  $\mathcal{A}$ and $\mathcal{B}$ are given we can consider a Poncelet Polygon as arising from the following process. 
We start with a point $p_1$ on  $\mathcal{A}$ and a tangent  $l_1$ to $\mathcal{B}$
through $p_1$. Then we choose the the  intersection of  $l_1$ with $\mathcal{A}$
other than $p_1$ and call it $p_2$. Through $p_2$ we draw the tangent to $\mathcal{B}$ other than $l_1$ and proceed iteratively.
Such a sequence of points $p_1,p_2,p_3,\ldots$ is called a Poncelet chain.
If the chain closes up after $n$ steps we call $p_1,p_2,\ldots,p_{n}$
a Poncelet $n$-gon.

The main statement of Poncelet's Porism is that if the construction closes up after $n$ steps for one starting point, then it will close up after $n$ steps for any starting point.
Figure \ref{fig:Mov1} illustrates the situation for  a star 7-gon.
A dynamic way to think about this theorem is to consider $p_1$
as gliding on the outer conic and by this creating a continuous motion of 
 the Poncelet polygon.

In  \cite{ BGRGT24a}  we showed that whenever we start with a Poncelet $n$ gon (with $n\geq 7$) its points can be used to construct$((k\cdot n)_4)$-configurations.
There we explicitly described the combinatorial
possibilities of getting such configurations. For details we refer to this companion article. The constriction can be considered as
starting with a Poncelet Polygon and then 
creating nested rings of certain star polygons.
Figure \ref{fig:Mov2} illustrates the resulting configuration for 
a Poncelet heptagon. The flexibility of the Poncelet configuration induces a flexibility on the resulting $((k\cdot n)_4)$-configuration.

\begin{figure}[H]
\begin{center}
\qquad\includegraphics[width=0.95\textwidth]{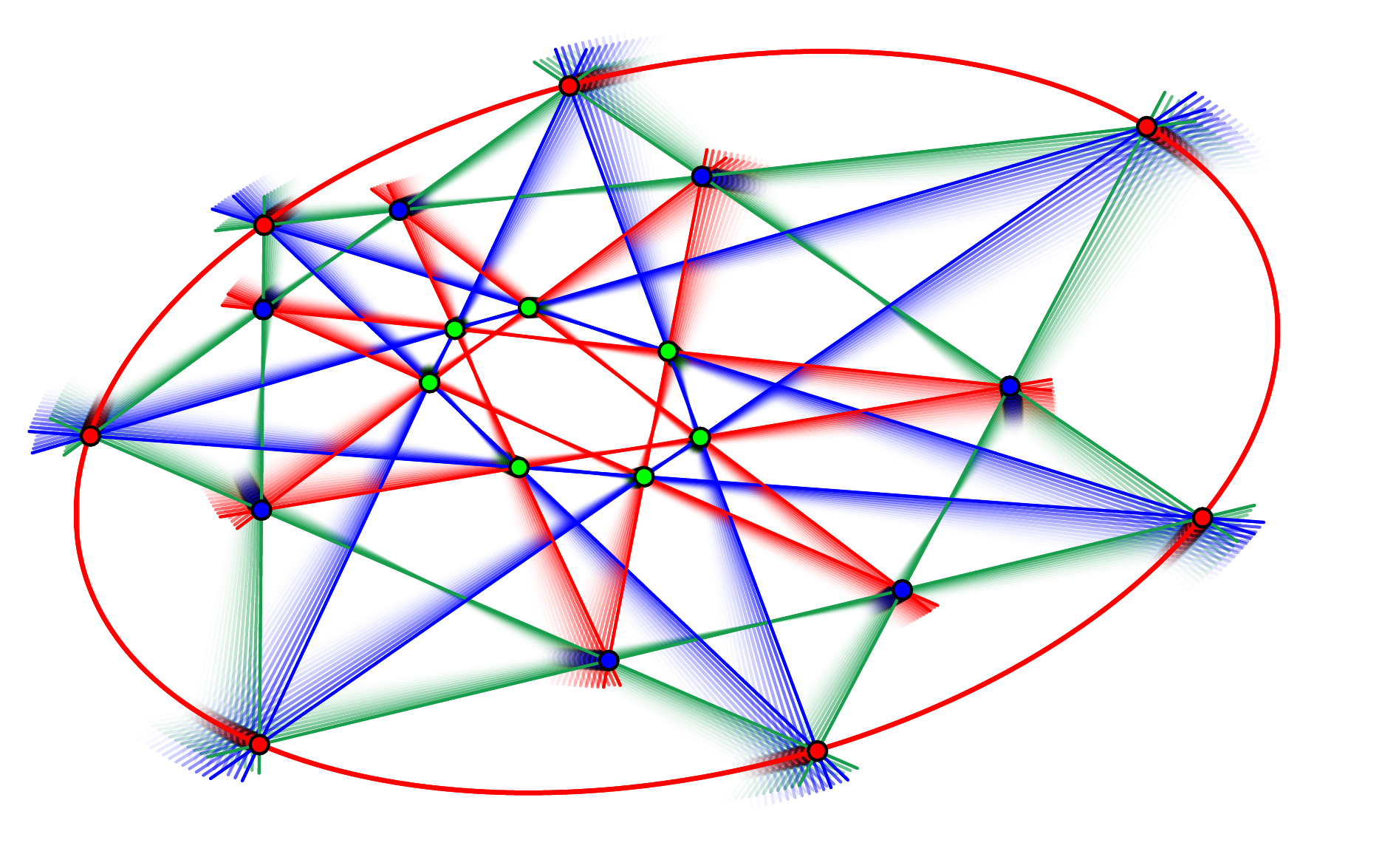}\
\begin{picture}(0,0)
\end{picture}
\end{center}
\captionof{figure}{An $(N_4)$-configuration in motion.}\label{fig:Mov2}
\end{figure}

If the points of the Poncelet $n$-gon are given, the corresponding 
$((k\cdot n)_4)$-configurations can be constructed by using join and meet operations only. Thus, providing geometric constructions for a Poncelet $n$-gon automatically 
leads to geometric construction of the corresponding $((k\cdot n)_4)$-configurations.

\subsection{Geometric primitives}
One of the aims of this article is to provide 
specific geometric constructions that lead to Poncelet $n$-gons for certain $n$. For this we must specify which construction tools exactly are allowed in that process.
For each geometric operation there is an algebraic counterpart.  To set the ground, we here specify the objects we will deal with and the relations between them that are relevant 
in our context. All our configurations may be interpreted in the real projective plane $\mathbb{RP}^2$ or in the complex projective plane $\mathbb{CP}^2$. We follow the standard 
treatment of geometric objects here. For now let $\mathbb{K}$ be either  $\mathbb{R}$ or  $\mathbb{C}$.   For the algebraic setup we follow  \cite{RG11} as well as \cite{Cox92,Cox94}.

\begin{itemize}
\item[]{\it Points} in the plane will be represented by 3-dimensional homogeneous coordinates 
$(x,y,z)^T\in \mathbb{K}^3-\{\mathbf{0}\}$. 
Vectors that differ only by a non-zero scalar multiple represent the same point.
\item[]{\it Lines} in the plane will also be represented by homogeneous coordinates $(a,b,c)^T\in \mathbb{K}^3-\{\mathbf{0}\}$. Vectors that differ only by a non-zero scalar multiple represent the same line.
\item[]{\it Incidence} between a point $p$ and a line $l$ arises when $p^Tl=0$. By this we can define operations {\it join} and {\it meet} via $l=p\times q$ for two points $p$ and $q$ and via
$p=l\times m$ {for two lines $l$ and $m$, respectively}. Although these two operations are algebraically identical, we represent them by different symbols. The join is written as $p\vee q$
and the meet as $l\wedge m$. 
\item[]{\it Conics} will be represented by symmetric $3\times 3$ matrices $\mathcal{A}$
that represent a quadratic form $p^T\mathcal{A}p$. A point $p$ is on the conic 
$\mathcal{A}$ if the quadratic form evaluates to zero. Matrices that differ only by a scalar multiple represent the same conic. Here we will only deal with non-degenerate conics. 
They satisfy $\det{\mathcal{A}}\neq 0$.
\item[]{\it Tangency} of a line $l$ and a non-degenerate  conic $\mathcal{A}$ is attained when $l^T\mathcal{A}^{-1}l=0$. 
We call $\mathcal{A}^{-1}$ the dual conic of $\mathcal{A}$.
\item[]{\it Intersections of lines with conics} can be calculated by solving algebraic equations. A line $l$ intersects a conic  
$\mathcal{A}$
in those points that simultaneously satisfy $p^T\mathcal{A}p=0$ and $p^Tl=0$. This requires solving a quadratic 
equation. 
\item[]{\it Tangents from a point to a conic} can also be calculated by solving algebraic equations. 
\item[]{\it Intersections of two conics} in general require the solution of a cubic equation. In \cite[Section 11.4]{RG11}  a specific method is given that performs this geometric operation. We will use it as a black box here.
\item[]{\it Projective transformations in the plane} map four points (in general position) to four points (in general position). They are represented by non-degenerate $3\times 3$ matrices $S$. 
A point $p$, a line $l$ and a conic $\mathcal{A}$ get mapped to 
$Sp$, $(S^{-1})^Tl$, $(S^{-1})^T\mathcal{A}S^{-1}$, respectively. 

\end{itemize}

\section{Bracket Polynomials for Poncelet chains} \label{sect:BracketPoly}

In this section we want to derive short algebraic conditions that encode  projective relations in the context of Poncelet polygons. 
Let $\mathcal{A}$ and $\mathcal{B}$ be two conics and $p_0\in \mathcal{A}$ a starting point for a Poncelet polygon with inscribed  conic $\mathcal{B}$ and circumscribed conic $\mathcal{A}$.
Recall that we call the sequence $p_1,p_2,p_3,p_4,\ldots$ that arises from successively forming tangents to $\mathcal{B}$ and intersections with $\mathcal{A}$ a {\it Poncelet chain}. 
If this sequence closes after $n$ steps and then traverses the same sequence of points, we call it a Poncelet $n$-gon (i.e.\ we have $p_i=p_{n+i}$ for all $i\geq 1$). 
If all points in a Poncelet $n$-gon are distict, then we call it a proper Poncelet $n$-gon.
In particular, non-proper $n$-gons may arise if the chain closes for a number smaller than $n$. 
This may happen if $k$ is a divisor of $n$, and the sequence closes after $k$ steps. However, this is not the only reason for being non-proper. For example, for proper  $n$-chains we require that no two points coincide.
Being a Poncelet $n$-gon is a projectively invariant property of the points involved.
One might be tempted to assume that starting with another point  on the conic $\mathcal{A}$ will only produce a projectively equivalent Poncelet polygon, but this is not the case. 
In general the Poncelet $n$-gons arising from $\mathcal{A}$ and $\mathcal{B}$ will be not projectively equivalent and exactly this is what makes the Poncelet Porism fascinating and also difficult to prove.

Projectively invariant properties of point sets can (by the first fundamental theorem of projective invariant theory)  always be expressed
by polynomial expressions in determinants of the homogeneous coordinates of the points. We denote the determinant of the $3 \times 3$ matrix whose columns are the homogeneous coordinates of points $p$, $q$, $r$ by the bracket $[p, q, r]$.
For instance, the fact that 6 points $1,2,3,4,5,6$ lie on a conic may be expressed by the bracket equation
\[
[1,2,3][1,5,6][4,2,6][4,5,3]=[4,5,6][4,2,3][1,5,3][1,2,6].
\]

This is an equation that may be considered as a relation in the projective plane $\mathbb{RP}^2$.
For our purposes we can make life even simpler by expressing conditions on the projective line $\mathbb{RP}^1$.
In order to see this, make the following observation. If we have a Poncelet chain $p_1,p_2,\ldots$, then all points necessarily lie on a conic $\mathcal{A}$.
We can now stereographically project this conic to a line and identify points on the conic with the points on that line. 
The center of projection (which lies on the conic) thus becomes the point at infinity on the line. We can literally identify the set of all points on $\mathcal{A}$  with $\mathbb{RP}^1$. 
Projective transformations of  $\mathbb{RP}^2$ that leave $\mathcal{A}$ invariant as a whole result in projective transformations of the corresponding~$\mathbb{RP}^1$. 
For a detailed treatment, see \cite[{Chapter 10}]{RG11}.

\medskip

By this reasoning, the fact that $p_1,p_2,\ldots$ forms a Poncelet chain becomes a projectively invariant property of the points in the underlying $\mathbb{RP}^1$. 
As a consequence of the fundamental theorem of invariant theory, this condition must be expressible as a condition on the $2\times 2$ determinants between 
homogeneous coordinates of the points in that $\mathbb{RP}^1$, that is, by a bracket equation. Similarly, the condition that 7 points form a Poncelet heptagon must also lead to bracket equations. 
This is what we are targeting as our next aim.

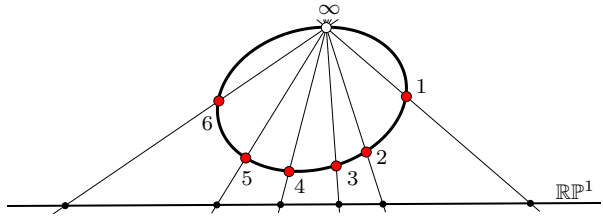
\begin{figure}[H]
\begin{center}
\includegraphics[width=0.65\textwidth]{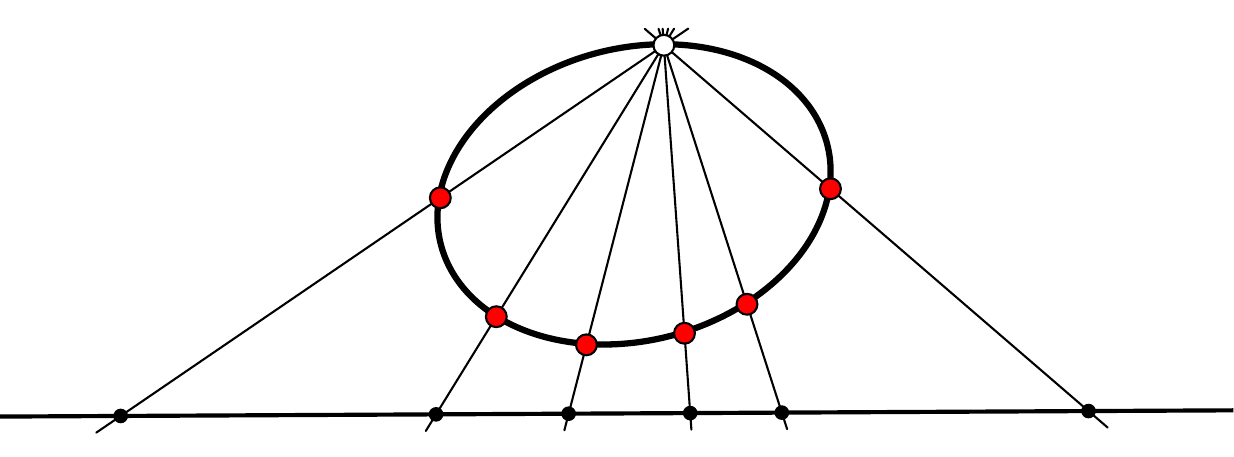}
\begin{picture}(0,0)
\put(-110,80){\footnotesize $\infty$}
\put(-73,50){\footnotesize $1$}
\put(-88,23){\footnotesize $2$}
\put(-99,15){\footnotesize $3$}
\put(-119,12){\footnotesize $4$}
\put(-139,15){\footnotesize $5$}
\put(-154,36){\footnotesize $6$}
\put(-20,10){\footnotesize $\mathbb{RP}^1$}

\end{picture}
\end{center}
\captionof{figure}{Stereographic projection of a conic to a line.}
\end{figure}

\subsection{Seven points forming a Poncelet Chain}

In what follows, for better readability, we will  identify points with their indices. So instead of $p_1$ we simply write $1$.
We will also identify points on a conic with the corresponding points in the related $\mathbb{RP}^1$. Furthermore, we denote
determinants of homogeneous coordinates of the corresponding points by a bracket $[\ldots]$ that lists the indices of the points. 
If no confusion can arise, we will even omit the commas to get a compact notation.
Thus, by the innocent abbreviation $[12]$ we mean the $2\times 2$ determinant of the points in the $\mathbb{RP}^1$
that are associated to two points labeled $p_1$ and $p_2$ on a conic. As such, these values are not well defined, and depend on the specifically chosen embedding. 
However, since we will only consider projectively invariant relations between points, the corresponding equations will be multihomogeneous 
bracket expressions in the points, and they turn out to be well-defined.

\medskip

For instance, the equation $[13][24]=-[14][23]$ represents the fact that $1,2,3,4$ are in harmonic position, and the value
$[13][24]/[14][23]$ 
corresponds to the cross ratio of the four points. 
Both expressions are projectively invariant, and will be well-defined in our setting (for details on this way of thinking, see \cite{RG11,RG95}).

Before we start, we recall the following fact, known as {\it Hesse's transfer pinciple} (see  \cite[\S 10.5 - 10.7]{RG11} for more details).
An ordered collection of points $(1,4: 2,5: 3,6)$ is called a {\it quadrilateral set} ({\it quadset}, for short)
if it arises from the intersection of that line with all lines that arise as the joins of 4 distinct points (see the lower part of Figure~\ref{fig:quad}). 
If five points in a quadset are given, the position of the sixth point is determined uniquely. 
Since `being a quadset' is invariant under projective transformations, it can be expressed by a corresponding bracket polynomial. 
It turns out that the corresponding algebraic expression for the points on the line is the relation
\begin{equation}\label{eq:quad}
[15][26][34]=[16][24][35].
\end{equation}

Now Hesse's transfer principle relates quadrilateral set relations to incidence properties of points on a conic. 
Given 6 points $1,2,3,4,5,6$ on a conic
with the property that the lines 14, 25, 36 meet in a point,  Hesse's transfer principle
implies that the point pairs $(1,4: 2,5: 3,6)$ form a quadset.
We will make use of that relation to derive our desired characterizations.

\medskip
The situation is illustrated in  Figure \ref{fig:quad} below. Six points on a conic are shown that are stereographically projected to a line.
The line plays the role of an $\mathbb{RP}^1$. The six points are located such that $14,25,36$ intersect  in a single point (as in the picture), and the six corresponding points on the line form a quadrilateral set. The green points and lines form a configuration that is a geometric certificate for the fact that $(1,4: 2,5: 3,6)$ is a quadrilateral set.

\begin{figure}[H]
\begin{center}
\includegraphics[width=0.65\textwidth]{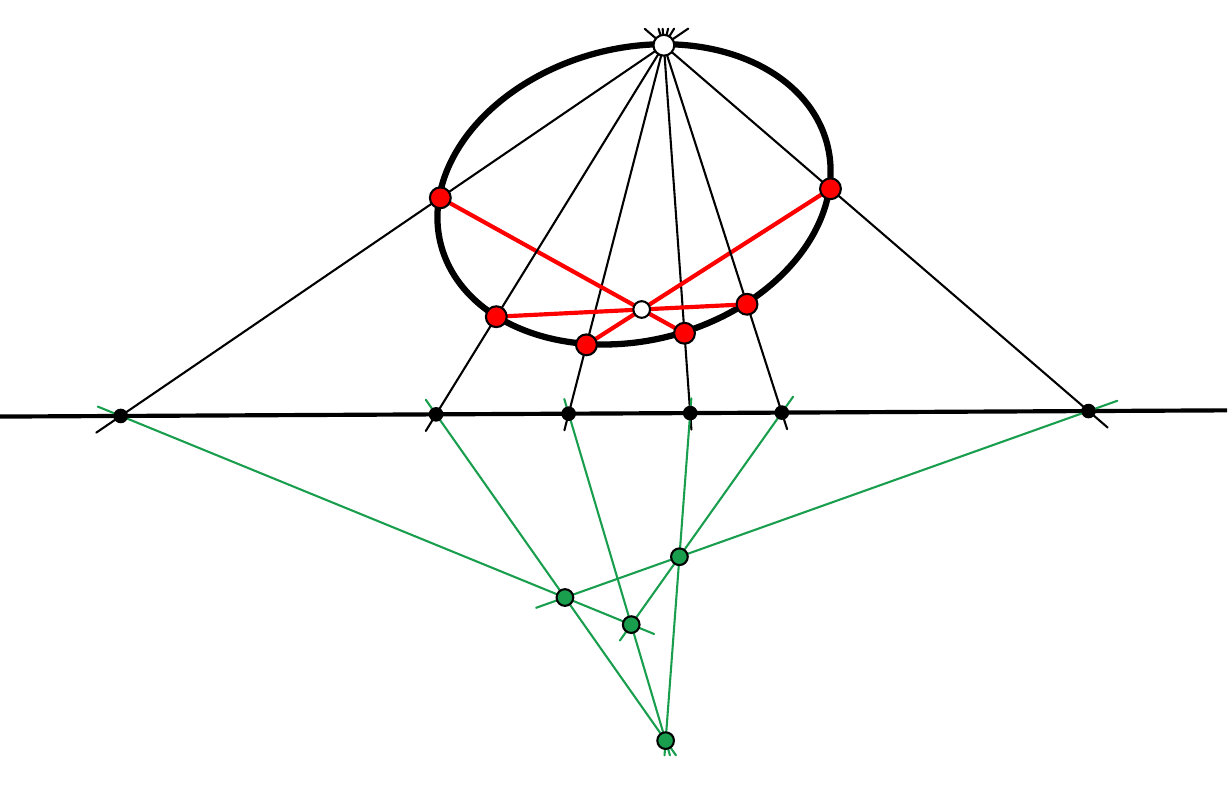}
\begin{picture}(0,0)
\put(-110,145){\footnotesize $\infty$}
\put(-70,110){\footnotesize $1$}
\put(-85,85){\footnotesize $2$}
\put(-99,77){\footnotesize $3$}
\put(-119,74){\footnotesize $4$}
\put(-139,78){\footnotesize $5$}
\put(-155,98){\footnotesize $6$}
\put(-20,75){\footnotesize $\mathbb{RP}^1$}

\end{picture}
\end{center}
\vskip-5mm
\captionof{figure}{Stereographic projection and Hesse's transfer principle.}\label{fig:quad}
\end{figure}
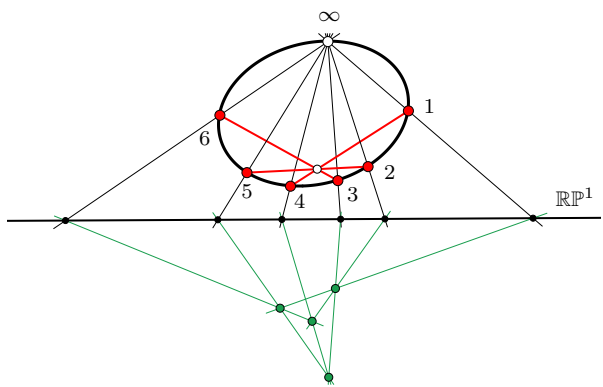

Now let $1,2,3,4,5,6,7$  be seven consecutive points in a proper Poncelet chain.
We want to develop a bracket expression that corresponds to the fact that the the lines $12, 23, 34, 45, 56, 67$ are simultaneously tangent to a conic. This is a projectively invariant property. 
The situation can be characterised by Brianchon's theorem stating that the long diagonals of the hexagon supported by these edges meet in a  point.

To characterize this algebraically, we need one more point that we do not yet have in the Poncelet chain: the intersection of 12 and 67.
Let $q$ be that point. Brianchon's condition now reads: the lines
 $4q; 52; 36$ meet in point.

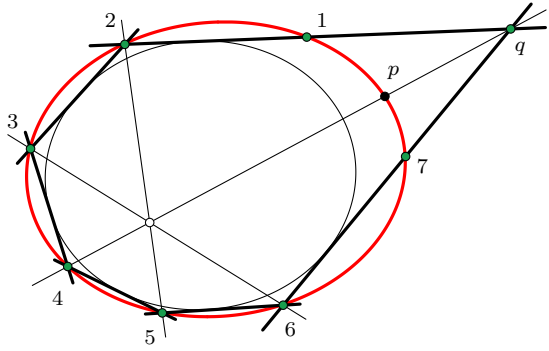
\begin{figure}[H]
\begin{center}
\includegraphics[width=0.65\textwidth]{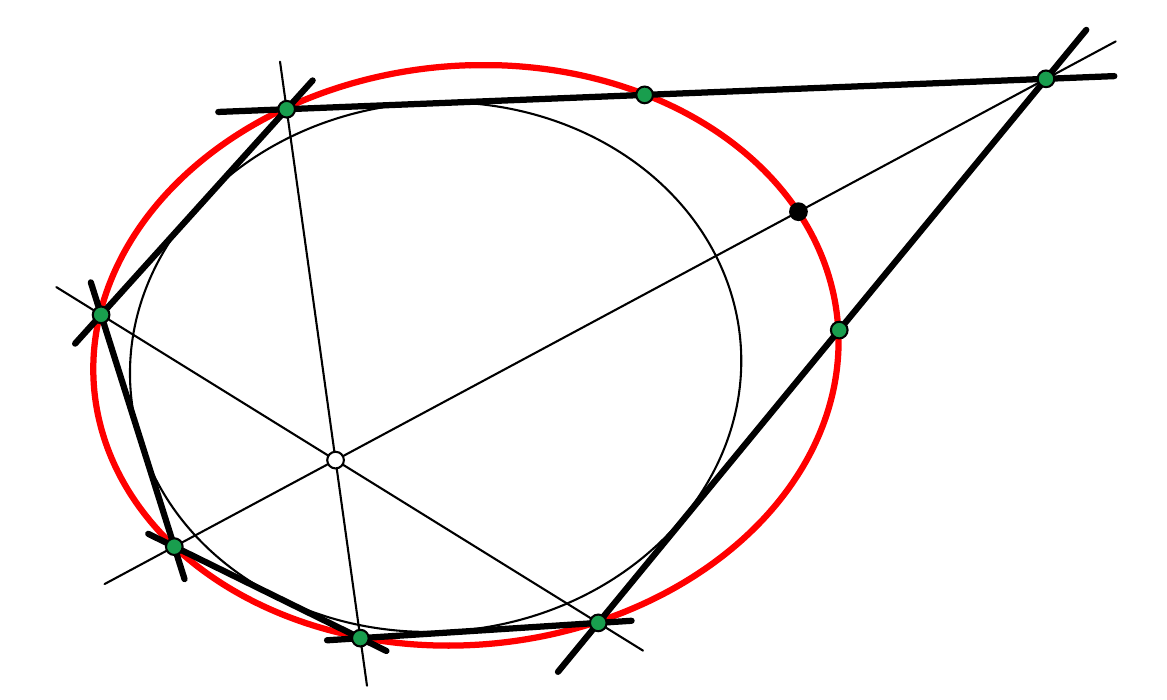}
\begin{picture}(0,0)
\put(-100,120){\footnotesize $1$}
\put(-180,120){\footnotesize $2$}
\put(-217,82){\footnotesize $3$}
\put(-200,15){\footnotesize $4$}
\put(-165,0){\footnotesize $5$}
\put(-112,3){\footnotesize $6$}
\put(-62,65){\footnotesize $7$}
\put(-73,102){\footnotesize $p$}
\put(-25,111){\footnotesize $q$}
\end{picture}
\end{center}
\captionof{figure}{A Poncelet 7-chain.}
\label{fig:Poncelet7Chain}
\end{figure}

Now consider the point $p$, the intersection (the one that is not 4) of $4q$ with the conic in Figure \ref{fig:Poncelet7Chain}.
All lines in the picture are spanned by the points $1,2,3,4,5,6,7,p$.
We want to have two concurrences between these lines.
One stating the existence of point $q$, which is the meet of $67; 21; 4p$.
And one expressing Brianchon's theorem, which is that $4p;52;63$ meet.
In terms of determinants, the conditions read (by formula (\ref{eq:quad})):

\[
\begin{array}{c}
[p2][74][16]=[p6][72][14]\phantom{.}\\[2mm]
[p6][54][32]=[p2][56][34].\\[2mm]
\end{array}
\]
After multiplying the left and right sides, one can cancel every determinant that contains $p$ and get:

\begin{equation}\label{eq:7chain}
[74][16][54][32]=[72][14][56][34].
\end{equation}

\medskip

This is a bracket condition which is automatically satisfied if $1\ldots 7$ forms a proper Poncelet chain in that order.
(One might ask if it was admissible to divide by the brackets $[p2]$ and $[p6]$, since we might have divided by zero.
However, in that case the point $p$ would coincide with $2$ or $6$, and then the chain would no longer be proper, since this also forces other points to coincide).
Observe that except  point $4$, every point occurs linearly in that expression.
Point $4$ is quadratic. 
Since the equation is linear in 7, this  can be applied in a very computational sense.
If points $1,\upto 6$ are given, this expression allows us to uniquely determine the position of the next point~$7$ in the Poncelet chain. 
Thus, we directly get a recursive algorithm to calculate Poncelet chains. That is:

\begin{theorem} \label{thm:ProperPonceletChain}
Seven distinct  consecutive  points $1,2,3,4,5,6,7$ form a proper Poncelet chain if and only if \[[74][16][54][32]=[72][14][56][34].\]
\end{theorem}

\medskip 

\noindent
A few comments are appropriate here:
One may read this equation as
\[
{[74][32]\over [72] [34]} =  {[14][56]\over [16][54]}.
\]

\noindent
On the left and right side one now has  two cross ratios. One can compactly denote this as

\[(7,3 ; 4,2) = (1,5 ; 4,6).\]

Since one can permute the letters in the cross ratios, one gets several other
additional equivalent expressions.
They translate to other bracket expressions. If we sort everything lexicographically,
we get the following three equivalent bracket characterizations for $7$-chains:
\[
\begin{array}{c}
[14][27][34][56]=[16][23][45][47],\\[2mm]
[14][24][37][56]=[15][23][46][47],\\[2mm]
[15][27][34][46]=[16][24][37][45].\\[2mm]
\end{array} 
\]
Two of them imply the third by multiplication and cancellation.
Each of them can also be expressed by equations relating two cross ratios.

\subsection{Six consecutive points in a Poncelet 7-gon}

Now we can combine these expressions to get the desired condition for six consecutive points in a proper Poncelet
heptagon. Let  $1,2,3,4,5,6,7$ be the points on a conic   in this cyclic order  forming a proper Poncelet 7-gon.
If the first six points are given, then the position of the point 7 is uniquely determined, since the points form a Poncelet chain.
Equivalently, one could think of this in the following way: Points $1\upto 6$ determine 5 segments. These segments uniquely
determine the inner conic, and this conic uniquely determines how the Poncelet chain proceeds. Hence, being a Poncelet $n$-gon
 imposes a projectively invariant condition on the position of the first $6$ points. We want to derive this condition for the Poncelet heptagon.

Depending on which point we start with, we get 7 different Poncelet chains. Each of them
creates 3 expressions like we had above. Hence, from the 7-gon we get all in all 21 (highly dependent) 
bracket expressions (each of them of the form $[ \ldots ][\ldots  ][\ldots  ][\ldots  ]=[\ldots  ][\ldots  ][\ldots  ][\ldots  ]$).

We are looking for a subset of them that can eliminate one point of the heptagon completely (say 7).
We can achieve that by the following choice:
Take the Poncelet chains $3,4,5,6,7,1,2$ and $5,6,7,1,2,3,4$ in the heptagon. From them one can deduce (by formula (\ref{eq:7chain}))

\[
\begin{array}{c}
[36][24][56][17]=[13][45][67][26],\\[2mm]
[35][67][12][14]=[15][46][17][23].\\[2mm]
\end{array} 
\]

Multiplying the left and right sides and performing cancellations, point 7 disappears completely, and one gets:

\begin{equation}\label{eq:7gon}
[36][24][56][35][12][14]=[13][45][26][15][46][23].
\end{equation}

This is the desired condition for 6 points being the first six points of a Poncelet 7-gon. Division by the brackets is admissible,
because we were starting with a proper heptagon. 
In this equation, each point is quadratic. In fact, this corresponds to the fact that after choosing the first five points,
there are exactly two possible positions of point 6 that lead to a Poncelet heptagon.
\medskip

\noindent
 Summarising, we get:

\begin{theorem}
Six points $1,2,3,4,5,6$ form six consecutive points in a proper Poncelet heptagon, if and only if \[[36][24][56][35][12][14]=[13][45][26][15][46][23].\]
\end{theorem}

\bigskip
\begin{remark}{\rm 
There are 7 Poncelet chains in the heptagon. Each of them produces 3 equations.
From these 21 equations one may ask how many of them are independent. It turns out
that if one knows the right 6 of them, the rest follows (actually this is asking for linear relations in the exponents of 
the monomials). Showing that a certain equation can be concluded may result in simply looking for linear
dependencies in the exponent vectors.
It is a straightforward calculation to check that the resulting ``rank'' of the space of equations is indeed six.
}
\end{remark}
\medskip

\begin{remark}{\rm
We derived the expression
\[[36][24][56][35][12][14]=[13][45][26][15][46][23]\]
as a condition. This expression may be considered as relation between cross ratios. This can be done in multiple ways.
One can regroup the equation as
\[
{\color{red}[36]}
{\color{OliveGreen}[35]}
{\color{blue}[56]}
{\color{red}[14]}
{\color{OliveGreen}[24]}
{\color{blue}[12]}
=
{\color{red}[13]}
{\color{OliveGreen}[23]}
{\color{blue}[26]}
{\color{red}[46]}
{\color{OliveGreen}[45]}
{\color{blue}[15]},
\]
and rewrite
\[
{\color{red}[36][14]\over[13][46]} \cdot {\color{OliveGreen}[35][24]\over [23][45]} \cdot{\color{blue}[56][12]\over [26][15]} =1,
\]
which reads as requiring the product of the cross ratios to equal 1:
\[
{\color{red}(1,6 ; 4,3)} \cdot {\color{OliveGreen}(3,4 ; 5,2)} \cdot{\color{blue}(5,2 ; 6,1)}=1.
\]
}
\end{remark}

\bigskip

\subsection{How many solutions are there?}\label{sectalgebra}
Before we continue, let us have a closer look at the algebraic nature of the 7-chain equation and algebraic conditions for $n$-gons for small $n$.
We have proved that the condition
\[[14][27][34][56]=[16][23][45][47]\]
is satisfied if the 7 points form a proper Poncelet chain. If we restrict to {\it proper} Poncelet chains, the condition is both necessary and sufficient, since point 7 is uniquely determined 
by the condition of being a Poncelet chain as well as by the algebraic condition.
The situation becomes slightly more complicated if we also allow non-proper choices of chains. Although the above equation may not be used to calculate the position of point 7 in certain 
(rare) degenerate situations, this only happens in situations where $1\upto 6$ do not uniquely determine 7.

We saw that we can use the 7 chain equation to derive the algebraic degree of the condition for point 6 under which we get a Poncelet $n$-gon. 
For $n=7$ this leads to a quadratic equation that has to be solved to determine the position of point 7. What happens for larger $n$? 
Since for the rest of the article it is not of high relevance, we only briefly describe what happens in the case of a concrete example.
Let us consider again the points represented on the projective line. We want to do the calculation for a specific, reasonably generic, choice of the first five points and the case $n=8$. 
This time we go ``down" to the level of very concrete coordinates. Let us assume that the point $i$ is given by homogeneous coordinates $(x_i,1)$ in $\mathbb{RP}^1$. 
A determinant $[i,j]$ then becomes $x_i-x_j$,
and our above equation now reads 
\[(x_1-x_4)(x_2-x_7)(x_3-x_4)(x_5-x_6)=(x_1-x_6)(x_2-x_3)(x_4-x_5)(x_4-x_7).\]
Up to projective transformation we may assume $x_1=-1$, $x_2=0$, $x_3=2$. Let us calculate the concrete situation $x_4=4$, $x_5=5$, and let us determine all positions for point $x_6$ 
that lead to a Poncelet octagon. We start with the equation
\[(-1-4)(0-x_7)(1-4)(5-x_6)=(-1-x_6)(0-1)(4-5)(4-x_7).\]
Solving for $x_7$ gives $x_7=(2 + 2x_6)/(38 - 7x_6)$.
Now we can calculate $x_8$ from $x_2\upto x_7$ by applying once more the 7-chain equation. We get
\[x_8={(5 (-8 + x_6) (-4 + x_6))\over(40 + x_6 (-156 + 29 x_6))}.\]
Performing one more iteration to calculate $x_9$, we obtain
\[x_9={(4 (-5 +  x_6) (100 +  x_6 (-108 + 17  x_6)))\over(-496 + 
 a (572 +  x_6 (-251 + 31  x_6)))}.\]
 Closing the chain by setting $x_9=x_1$, we obtain the equation:
 \[-2496 + 3132 x_6 - 1023 x_6^2 + 99 x_6^3=0.\]
Amazingly, this polynomial factors to
\[( 624 - 627 x_6 + 99 x_6^2)\cdot(x_6-4).\]
So we get one solution $(x_6)_1=4$ and two solutions $(x_6)_{2,3}={209 \pm 5 \sqrt{649}\over 66}.$
Now a little care is appropriate. These are all solutions for which we get $x_9=x_1$; however, it is not a priori clear that each of these solution   indeed creates a Poncelet 8-gon. 
We must have $x_i=x_{i+8}$ for {\it every} $i$. The following table shows what happens if we calculate the first 10 points in a chain for the different solutions of $x_6$:
\[
\begin{array}{rcl}
x_6=4&\to&-1, 0, 1, 4, 5, 4,1, 0, -1, 10\\[2mm]
x_6={209 - 5 \sqrt{649}\over 66}&\to&
-1, 0, 1, 4, 5, {209 -5 \sqrt{649}\over 66}, {27 -\sqrt{649}\over 10} , 
{105 -5\sqrt{649}\over 26}, -1, 0\\[2mm]
x_6={209 + 5 \sqrt{649}\over 66}&\to&
-1, 0, 1, 4, 5, {209 + 5 \sqrt{649}\over 66}, {27 +\sqrt{649}\over 10} , 
{105 +5\sqrt{649}\over 26}, -1, 0.\\[2mm]
\end{array}
\]
While the two solutions of the quadratic equation produce a solution with $x_{10}=x_2=0$ (as it should), the other solution $x_6=4$
produces $x_{10}=10\neq x_2$, and thus is not a Poncelet chain. Hence for the Poncelet octagon there are exactly two solutions 
that can be derived, by solving a quadratic equation. Since the choice of $x_4$ and $x_5$ was reasonably generic, we can expect similar behaviour for other choices as well.

\medskip

With analogous calculations we can find the expected number of solutions (and their algebraic degree) for other $n$.
We only quote the list of the first few here:
\[
\begin{array}{c||c|c|c|c|c|c|c|c|c|c|c|c|c|c|c} 
n&6&7&8&9&10&11&12&13&14&15&16&17&18&19&20\\
\hline
\mathrm{solutions}&1&2&2&3&4&5&5&7&8&9&10&12&13&15&16\\
\end{array}
\]

In particular, every $n>8$ requires the calculation of a root of a polynomial of degree at least 3.

\section{Geometric Construction of Poncelet $7$-gons and $8$-gons} \label{sect:ConstructionPoncelet}

The  main  target of this section is to 
provide explicit elementary geometric constructions that 
can create all Poncelet  $7$-gons and all Poncelet $8$-gons.
Furthermore, we will prove that these constructions  indeed create the desired Poncelet polygon.

Before we start, let us do a general count of {\it degrees of freedom}, and by this get an impression of the dimension of the 
related configuration spaces.

In what  follows we consider the term $n$-gon in its broadest sense. 
We will allow self-intersecting situations, and represent the segments by the lines spanned by their endpoints.
A Poncelet $n$-gon is a sequence $(p_1,p_2,\ldots, p_n)$ of points which lie on a common conic $\mathcal{A}$
where the lines  $p_ip_{i+1}$ are each tangent to a second conic $\mathcal{B}$.

The property of being a Poncelet $n$-gon is a projective invariant: Every projective transformation of a Poncelet $n$-gon
is again a Poncelet $n$-gon. Thus we may choose the first 4 points $p_1,p_2,p_3,p_4$ of a Poncelet $n$-gon arbitrarily to
form a projective basis. The next point $p_5$ may be chosen arbitrarily as well as long as we require that no triple of these 5 points is collinear. This can be seen by the following 
argument. 
Five points  $p_1,\ldots, p_5$ with no triple collinear determine a unique non-degenerate conic $\mathcal{A}$.
Thus, all remaining points will also be on conic $\mathcal{A}$. If we choose a point $p_6$ on that conic, the five  lines $p_1p_2,\ldots, p_5p_6$ determine the inner conic $\mathcal{B}$. 
Whether the Poncelet chain created by these 6 points closes up after $n$ steps is an algebraic condition on the position of $p_6$. If we express the position of $p_6$ in terms of the 
1-dimensional coordinates of the $\mathbb{RP}^1$ that can be associated to $\mathcal{A}$, then the possible positions of $p_6$ are the roots of a suitably chosen polynomial.
Thus,  after modding out projective transformations we have $2$ continuous degrees of freedom (the position of $p_5$) and after that a discrete choice of positions for the remaining points. 
We want to emphasise that this situation is independent from $n$ as long as $n\geq 5$. We may also think of this as starting with an arbitrary non-degenerate conic and then choosing 5 different points on it, which uniquely determine finitely many possible Poncelet $n$-gons for every $n>5$.

\subsection{Poncelet triangles to hexagons}
We briefly remark that it is fairly  simple to construct Poncelet $n$-gons for $n=3,4,5,6$. 
For $n=3,4,5$ any choice of $n$ distinct points corresponds to a Poncelet polygon. The inner conic can easily be chosen simultaneously tangent to all $n$ sides. 
For $n=6$ it requires one application of Brianchon's Theorem to determine an admissible position of point $p_6$ after the first five points are given. 
The situation becomes more complicated for $n\geq 7$.

\subsection{Poncelet heptagons}

Now we consider the particular case $n=7$. In the last section we concluded that (for the ${\mathbb{RP}^1}$ coordinates of the points on a conic $\mathcal{A}$)
we have (\ref{eq:7gon}):
\[
[36][24][56][35][12][14]=[13][45][26][15][46][23].
\]
(Again by abuse of notation we identify a point $p_i$ with its index). This is the condition that $p_6$ in the conic has to satisfy after
 $p_1\upto p_5$ are given.

 \medskip

\noindent
 The following construction yields a suitable point $p_6$ (see Figure~\ref{fig:7gon1}).

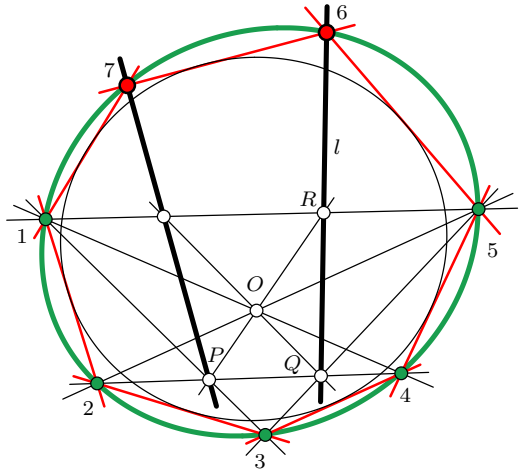
\begin{figure}[H]
\begin{center}
\includegraphics[width=0.65\textwidth]{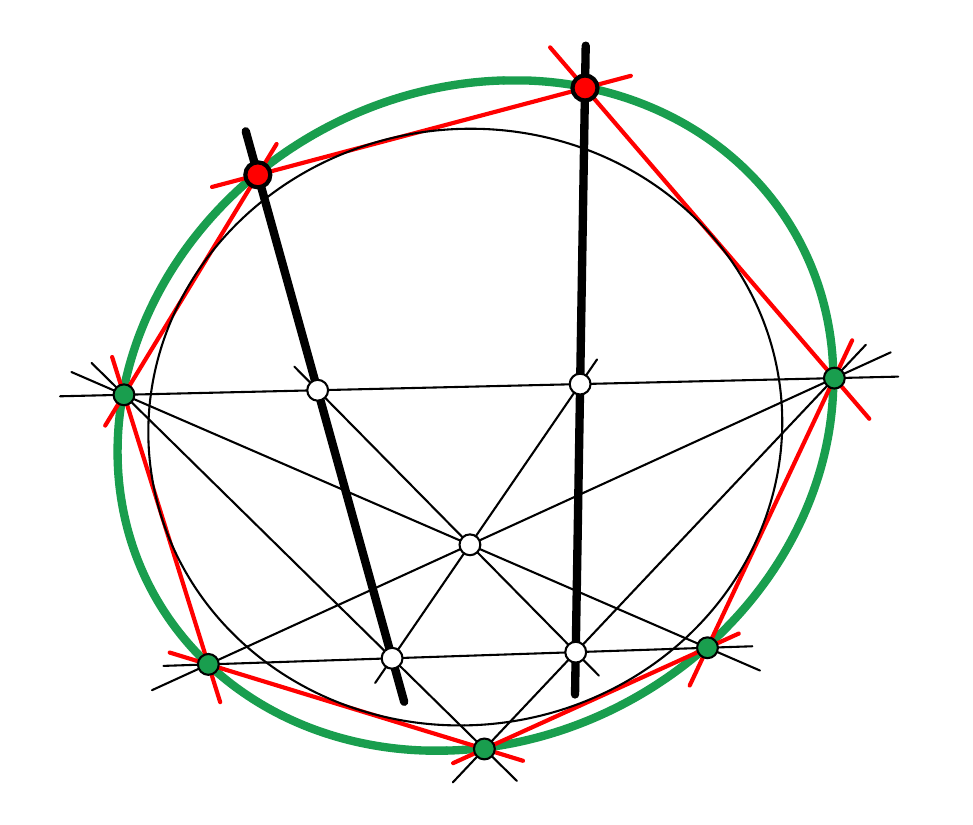}
\begin{picture}(0,0)
\put(-210,90){\footnotesize $1$}
\put(-185,25){\footnotesize $2$}
\put(-120,5){\footnotesize $3$}
\put(-65,30){\footnotesize $4$}
\put(-32,85){\footnotesize $5$}
\put(-89,175){\footnotesize $6$}
\put(-177,153){\footnotesize $7$}
\put(-123,72){\footnotesize $O$}
\put(-138,45){\footnotesize $P$}
\put(-109,42){\footnotesize $Q$}
\put(-103,104){\footnotesize $R$}
\put(-90,124){\footnotesize $l$}
\end{picture}
\end{center}
\captionof{figure}{Construction of a Poncelet 7-gon.}\label{fig:7gon1}
\end{figure}

\begin{construction}\label{const1}
{\rm
The start of the construction is five free points $1,\ldots, 5$ (no three collinear).
Then we construct
\begin{itemize}
\item[$O$:] intersection of $14$ and $25$,
\item[$P$:] intersection of $13$ and $24$,
\item[$Q$:] intersection of $24$ and $35$,
\item[$R$:] intersection of $15$ and $OP$,
\item[$l$:] join of $R$ and $Q$.
\end{itemize}

Finally, we intersect the line $l$ with the conic $\mathcal{A}$ that passed through the initial 5 points. We get two intersections and call one of them $6$. 
The chain $1,2,3,4,5,6$ will be the initial sequence of a Poncelet 7-gon.
}
\end{construction}

\noindent
Before we continue to prove this result a few subtleties should be mentioned.

\begin{itemize}
\item The choice of the point 6 is not unique. The line $l$ has {\it two} intersections with $\mathcal{A}$. Each of the 
choices  leads to a  Poncelet heptagon. This reflects the fact that the underlying algebraic condition is quadratic.
The resulting situation for the other choice is shown in Figure \ref{fig:7gon2}.

\item We did not explicitly construct point 7 in the above procedure. There are three ways to construct this point:
\begin{itemize}
\item [a)] algebraically:  The position of point 7 can be determined by the Poncelet chain condition from Theorem~\ref{thm:ProperPonceletChain};
\item [b)] geometrically:  There are some fairly simple constructions to obtain point 7 only by join and meet constructions from the previously constructed points $1\upto 6$, using variants of Pascal's or Brianchon's theorem; 
\item [c)] by symmetry:  One could also perform a symmetric version of the construction of point 6. However, one has to be aware that the choice of the  intersection 
of point 7 must be consistent with the choice of point 6.
\end{itemize}
\end{itemize}

\begin{figure}[H]
\begin{center}
\includegraphics[width=0.75\textwidth]{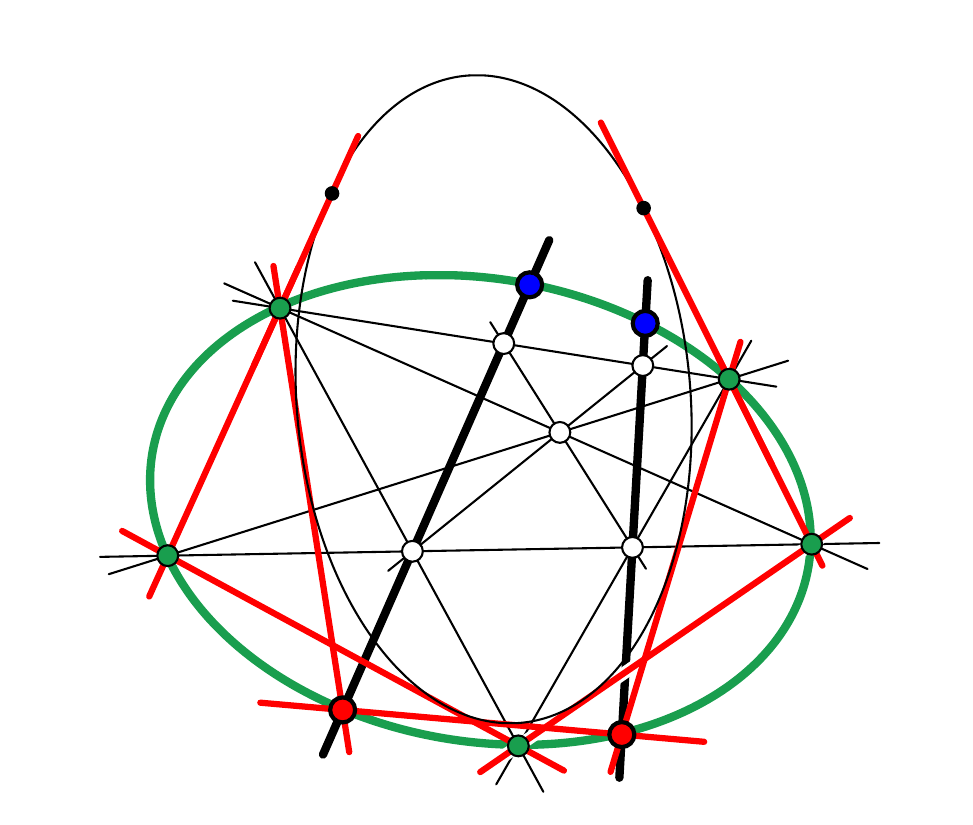}
\begin{picture}(0,0)
\put(-200,146){\footnotesize $1$}
\put(-220,58){\footnotesize $2$}
\put(-127,9){\footnotesize $3$}
\put(-45,83){\footnotesize $4$}
\put(-71,127){\footnotesize $5$}
\put(-92,12){\footnotesize $6$}
\put(-180,23){\footnotesize $7$}
\end{picture}
\end{center}
\captionof{figure}{The other possible choice. Here a self-overlapping Poncelet polygon is generated. Some of the touching points to the second conic lie in the periphery.} \label{fig:7gon2}

\end{figure}
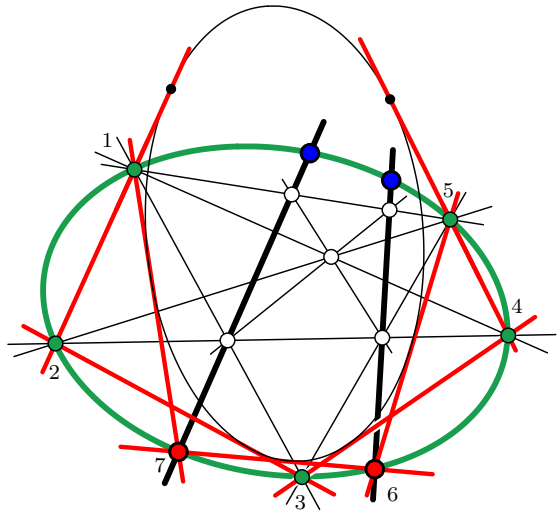

\noindent{\it Remark:}
After the points $1,\ldots, 7$ are constructed by Construction \ref{const1}
as Figure \ref{fig:7gon1} indicates, it is easy to create a Grünbaum--Rigby configuration $(21_4)$ from them as shown in Figure \ref{fig:Mov2}.
We count indices modulo 7 and define an operation on seven points 
$P=(p_1, \ldots p_7)$ by 
\[\vee_a(P):=(p_1\vee p_{1+a},\ldots,p_7\vee p_{7+a})\]
Dually,  for a configurations of lines $L=(l_1, \ldots l_7)$
we define  
\[\wedge_b(L):=(l_1\vee l_{1-b},\ldots,l_7\vee l_{7-b}).\]
Starting from the six points $P=(1,2,3,4,5,6,7)$ of our construction we 
construct all intermediate construction elements of
\[
P':=\wedge_3\vee_1\wedge_2\vee_3\wedge_1\vee_2(P).
\]
In \cite{BGRGT24a} we showed that if $P$ is a Poncelet polygon then $P=P'$ and creates the Grünbaum--Rigby configuration.

\begin{theorem}\label{constr7gon}
If none of the construction steps in Construction 1  becomes degenerate,
then it produces an admissible 6th point of a Poncelet heptagon.
\end{theorem}
\begin{proof}
In what follows we again take the perspective that we identify the conic $\mathcal{A}$ with a projective line 
$\mathbb{RP}^1$ via stereographic projection.
We represent points on the conic by 2-dimensional vectors. Each relevant line will be represented by
its two intersections with the conic. Three lines meeting in a point  can algebraically be translated into a quadset condition
over the $2\times 2$ determinants of the points.
We refer to Figure~\ref{fig:7gon3}  for the labelling. We labelled all relevant points and lines. 
Those points that lie on the conic $\mathcal{A}$ are labelled by the numbers $1,\ldots,9$.
Point $\overline{6}$ denotes the other possible intersection of line $l$ and the conic. 
There are two reasons for considering point $\overline{6}$.
Firstly, it is needed as an auxiliary point to express the concurrence of two other lines with line $6\overline{6}$ in terms of quadset relations. This is what is needed for our process of translating everything into projective condition over $2\times 2$ determinants.
Secondly since there is algebraically no way to distinguish $6$ from
$\overline{6}$ this point leads to the second possible solution for how we can complete $1,\ldots, 5$ to a Poncelet heptagon.

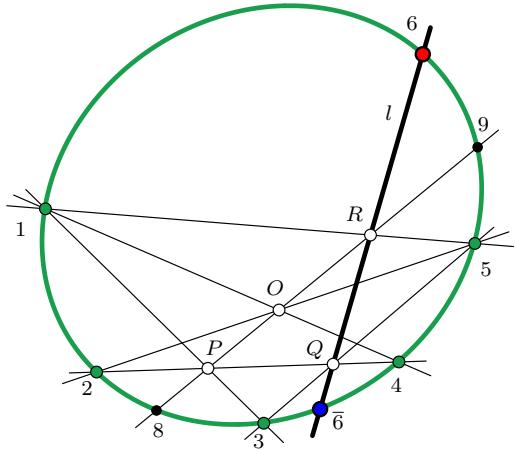
\begin{figure}[H]
\begin{center}
\includegraphics[width=0.60\textwidth]{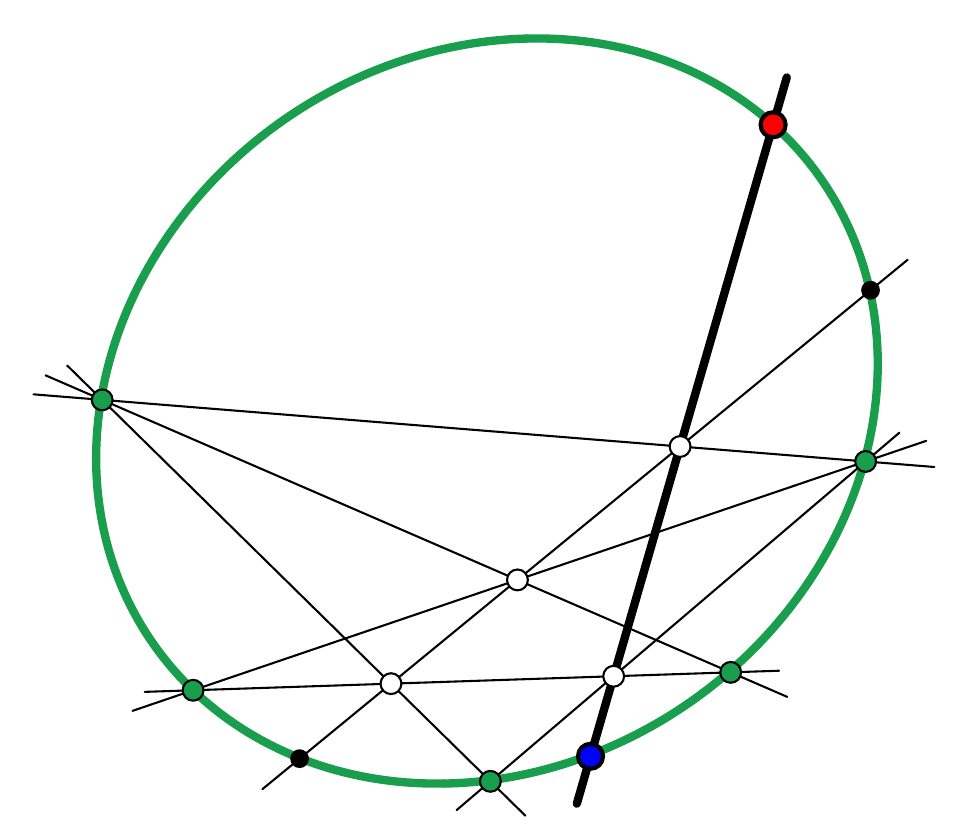}
\begin{picture}(0,0)
\put(-200,80){\footnotesize $1$}
\put(-175,20){\footnotesize $2$}
\put(-110,0){\footnotesize $3$}
\put(-58,21){\footnotesize $4$}
\put(-24,65){\footnotesize $5$}
\put(-52,158){\footnotesize $6$}
\put(-80,7){\footnotesize $\overline{6}$}
\put(-148,4){\footnotesize $8$}
\put(-25,120){\footnotesize $9$}
\put(-105,58){\footnotesize $O$}
\put(-128,35){\footnotesize $P$}
\put(-90,35){\footnotesize $Q$}
\put(-75,86){\footnotesize $R$}
\put(-60,124){\footnotesize $l$}
\end{picture}
\end{center}
\captionof{figure}{Points involved in the proof of Theorem \ref{constr7gon}}.\label{fig:7gon3}
\end{figure}

 Points 8 and 9 are the two intersections of line $OP$ with the conic. We will express all concurrences of three lines that happen in the interior of the conic 
by relations of points on the $\mathbb{RP}^1$-equivalent of the conic specifically:

\begin{itemize}
\item[$O$:] $(14\colon 25\colon 89)$ are a quadset,\\[-6mm]
\item[$P$:] $(13\colon 24\colon 89)$ are a quadset,\\[-6mm]
\item[$Q$:] $(24\colon 35\colon 6\overline{6})$ are a quadset,\\[-6mm]
\item[$R$:] $(15\colon 6\overline{6}\colon 89)$ are a quadset.
\end{itemize}
We want to show that these quadset relations together imply the
heptagon conditon:

\[[36][24][56][35][12][14]=[13][45][26][15][46][23].\]

\medskip
\noindent {\bf Step 1: eliminate 8 and 9}

\smallskip\noindent
Each of these quadset relations can be transferred in different ways into a bracket polynomial equation. We first 
concentrate on the condition coming from points $O,P,R$, {which}  are  the only conditions containing point $8$ and $9$.
We can eliminate those two points by multiplying the left and right sides of the following bracket equations,
and performing cancellation of those terms that occur twice.

\[
\begin{array}{lcrl}
[19]\cdot {\color{red}[54]} \cdot[82]&=& {\color{red}[12]}\cdot[59]\cdot[84]&\qquad (O)\\[1mm]
[18]\cdot  {\color{red}[54]}\cdot[92]&=& {\color{red}[12]}\cdot[58]\cdot[94]&\qquad (O)\\[3mm]
 {\color{red}[21]}\cdot[39]\cdot[84]&=&[29]\cdot {\color{red}[34]}\cdot[81]&\qquad (P)\\[1mm]
 {\color{red}[23]}\cdot[18]\cdot[94]&=&[28]\cdot{\color{red}[14]}\cdot[93]&\qquad (P)\\[3mm]
{\color{red}[1\overline{6}]}\cdot[95]\cdot[68]&=&[18]\cdot[9\overline{6}]\cdot{\color{red}[65]}&\qquad (R) \\[1mm]
{\color{red}[16]}\cdot[85]\cdot[\overline{6}9]&=&[19]\cdot[86]\cdot{\color{red}[\overline{6}5]}&\qquad (R)\\[1mm] 
\end{array}
\]

After multiplying all left and all right sides the only brackets that remain are the red ones,
and we are left with two equations

\[
\begin{array}{rcll}
[54]^2 \cdot  [21]\cdot[23]\cdot[1\overline{6}]\cdot[16]&=&[12]^2\cdot[34]\cdot[14]\cdot[65]\cdot[\overline{6}5]&\qquad (O,P,R)\\[1mm]
[23]\cdot[5\overline{6}]\cdot[64]&=&[2\overline{6}]\cdot[54]\cdot[63]&\qquad (Q)\\[1mm]
\end{array}
\]

\medskip
None of these equations contains 8 or 9 any more. The first equation is our combined knowledge about the points $O,P,R$.
The second equation expresses the  concurrence at point $Q$ using a carefully-chosen bracket condition.

\medskip
\noindent {\bf Step 2: eliminate $\overline{\mathbf{6}}$}

\smallskip\noindent
Next, we are going to eliminate point $\overline{6}$ so that in the end we have a condition that expresses the position of point 6 directly in terms of $1,\ldots,5$. 
Unfortunately, there seems to be no way to eliminate this point similar to the one that we used for 8 and 9. 
So, this time we have to use other measures that (unfortunately) increase the number of summands.

Since all relevant points are in $\mathbb{RP}^1$, we may express $\bar{6}$ as the linear combination of two other points. We choose $p_{5}$ and $p_{2}$, and write:
\[\overline{6}=\lambda\cdot 5 + \mu\cdot 2.\]
We insert this into 
$[23][5\overline{6}][64]=[2\overline{6}][54][63]$, and get
\[
[23][5(\lambda\cdot 5 + \mu\cdot 2)][64]=[2(\lambda\cdot 5 + \mu\cdot 2)][54][63].
\]
The multilinearity of the determinant and the fact that $[55]=[22]=0$ leaves us with
\[
\mu\cdot[23][52][64]=\lambda\cdot[25][54][63].
\]
Up to a common multiple (which is not a problem since we are working with homogeneous coordinates) this leaves us with the unique solution
\[
\lambda=[23][52][64];\quad \mu=[25][54][63].
\]
So, if we set:
\[\overline{6}=[23][52][64]\cdot 5 + [25][54][63]\cdot 2,\]
we will automatically satisfy the quadset relation around $Q$.
We now reformulate the longer expression $(O,P,R)$ with this special choice for $\bar{6}$. We get:
\[
\begin{array}{rcl}
&&[54]^2[21][23][1{\color{red}([23][52][64]\cdot 5 + [25][54][63]\cdot 2)}][16]\\[1mm]
&-&[12]^2[34][14][65][{\color{red}([23][52][64]\cdot 5 + [25][54][63]\cdot 2)}5]=0.
\end{array}
\]
The red parts are the ones that replaced point $\overline{6}$.
Expanding by multilinearity leaves us with:
\[
\begin{array}{rcl}
&&[54]^2[21][23][23][52][64][15][16]\\[1mm]
&+&[54]^2[21][23][25][54][63][12][16]\\[1mm]
&-&[12]^2[34][14][65][25][54][63][25]=0.
\end{array}
\]

\noindent
We can  eliminate a factor $[12][25][45]$, and get
\begin{equation}\label{eq:precond}
-[15][16][23]^2[45][46]-[12][16][45]^2[23][36]+[12][14][25][34][56][36]=0.
\end{equation}

This expression  no longer contains point $\overline{6}$. 
Thus this condition expresses the relation between the six points $(1,\ldots,6)$ in our construction.
If point 6 is created by Construction~\ref{const1}, it will satisfy this equation.
Observe that this equation is quadratic in 6 (as it should be), leading to two possible solutions.

\medskip
\noindent {\bf Step 3: equivalence to Poncelet 7-gon test}

\smallskip\noindent
After these reductions it remains to show that every solution of equation (\ref{eq:precond})
with given $1,\ldots,5$
leads to a valid sixth point for a Poncelet heptagon, i.e.\ to a solution of the test condition (\ref{eq:7gon}):
\[[36][24][56][35][12][14]-[13][45][26][15][46][23]=0.\]

In a certain sense this is almost a trivial fact since it turns out that the three summand polynomial that we derived in Step 2  is just literally our
testing condition in disguise. Due to general dependencies among determinant polynomials (sometimes called syzygies), the same 
polynomial may be expressed in different ways. It turns out that exactly this happens here.

One could easily finish the discussion here by a statement like ``After expanding both polynomials and cancelling vanishing terms,
both polynomials turn out to be identically the same". This could be easily done in a split second by a computer algebra program. 
It could be done even by hand, since each monomial only expands into $2^6=64$ terms. However, we want to be more specific here, 
and explicitly ``see'' the equivalence.

\medskip

The main reason for dependencies among determinants are the so-called {\it Grassmann--Plücker relations}.
For $\mathbb{RP}^1$, they are relations of the form
\begin{equation}\label{eq:gpr}
[ab][cd]-[ac][bd]+[ad][bc]=0.
\end{equation}
These relations hold for arbitrary points $a,b,c,d$, and it can be shown (Second Fundamental Theorem of projective invariant theory)
that every more-complicated relation comes from monomial linear combinations of Grassmann--Plücker relations.
Let us consider the following  linear combination of Grassmann--Plücker relations:
\[
\begin{array}{cccl}
0&=&&[15][45][46][23]\cdot([12][36]-[13][26]+[16][23])\\[1mm]
&&+&[12][23][36][45]\cdot([14][56]-[15][46]+[16][45])\\[1mm]
&&-&[12][36][14][56]\cdot([23][45]-[24][35]+[25][34]).\\[1mm]
\end{array}
\]
Expanding yields
\[
\begin{array}{c@{\ }c@{\ }c@{}l}
0&=&&{\color{red}[15][23][45][46][12][36]}
-[15][23][45][46][13][26]
+[15][23][45][46][23][16]\\[1mm]
&&+&{\color{blue}[12][23][36][45][14][56]}
-{\color{red}[12][23][36][45][15][46]}
+[12][23][36][45][16][45]\\[1mm]
&&-&{\color{blue}[12][36][14][56][23][45]}
+[12][36][14][56][24][35]
-[12][36][14][56][25][34].\\[1mm]
\end{array}
\]

\noindent
Terms with the same color cancel, and we are left with
\[
\begin{array}{c@{\ }c@{\ }l@{}l}
0&=&+
[15][23][45][46][23][16]
+[12][23][36][45][16][45]
-[12][36][14][56][25][34]
\\[1mm]
&&
-[15][23][45][46][13][26]
+[12][36][14][56][24][35]
.\\[1mm]
\end{array}
\]
The first row in this expression is exactly our construction polynomial
(\ref{eq:precond});
the second row is exactly the negative of our testing polynomial (\ref{eq:7gon}). 
Since the above expression is zero, the construction polynomial and the testing polynomial are equal. This finishes the proof of the theorem.
\end{proof}

Using the techniques developed in \cite{BGRGT24a}, Figure~\ref{fig:21_7H}
shows a realisation of the Grünbaum-Rigby $(21_4)$-configuration \cite{GrRi90} that is based
on an unusual distribution of points for the initial Poncelet polygon. The (red) points are chosen in such a way
 that the corresponding two conics that support the Poncelet polygon are intersecting in two real points. It is based on the choice of points shown in Figure~\ref{fig:7gon2}.

\begin{figure}[h]
\begin{center}
\includegraphics[width=0.80\textwidth]{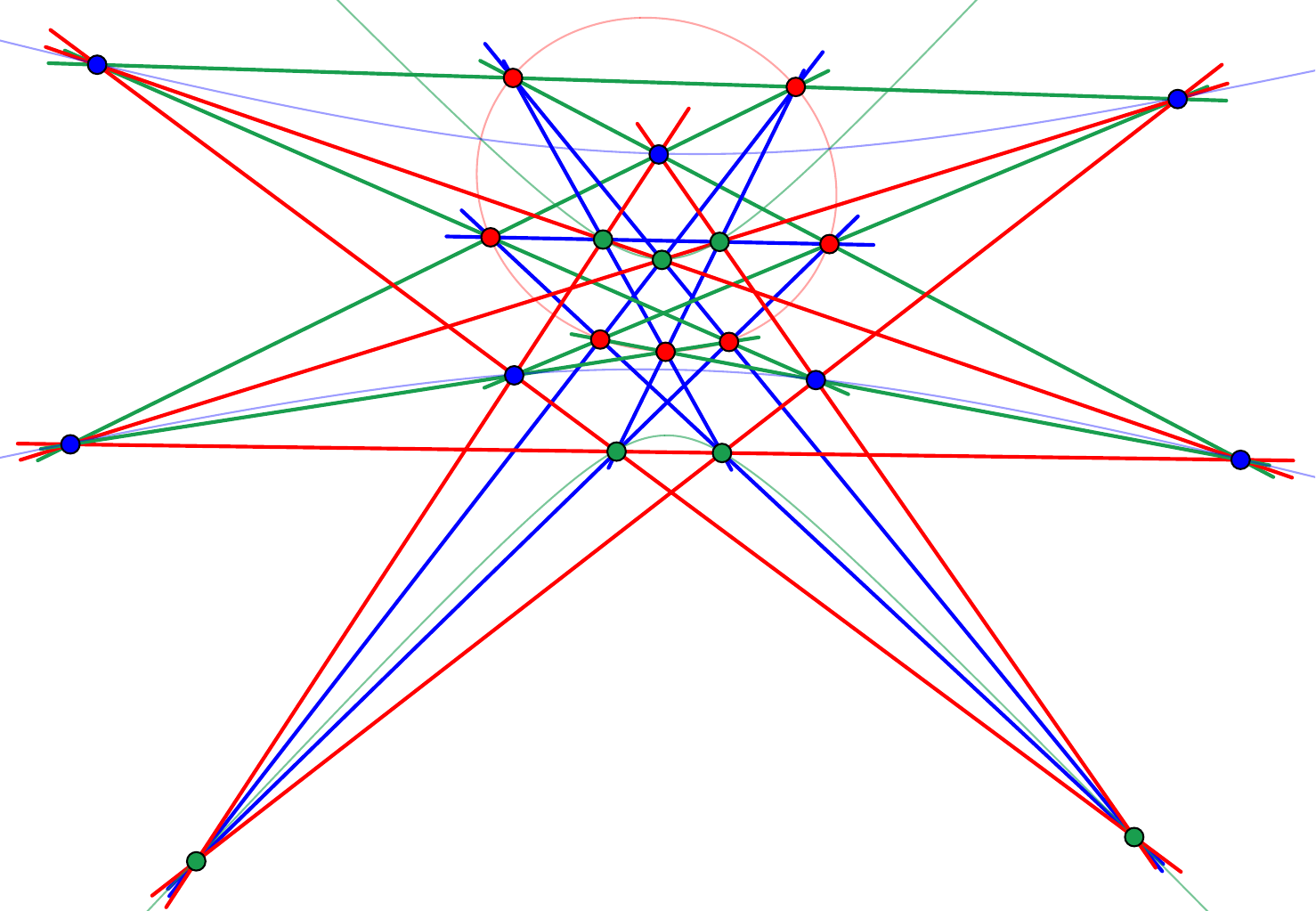}
\begin{picture}(0,0)
\end{picture}
\end{center}
\captionof{figure}{An unusal realization of $(21_4$). Supporting conics of the three rings of points are indicated faintly.}\label{fig:21_7H}
\end{figure}

\subsection{Poncelet octagons}

Let us now briefly take a look at the Poncelet octagon.
We will provide an explicit bracket condition, and give a geometric construction.
In a certain sense, both a bracket characterisation and a geometric construction are at the same time more difficult and easier than for the Poncelet 7-gon.
If we attempt to start with points $1,\ldots 5$, and then characterise and construct the possible positions of point $6$, then the construction becomes more involved than in the heptagon case,
and the characterisation in terms of bracket polynomials becomes more complicated.
There seems to be no way to characterise the geometric situation by a bracket polynomial that only has two summands (like it was possible to do in the 7-gon case).

\medskip
Fortunately, the Poncelet 8-gon possesses a richer structure than the 7-gon, and this allows us to derive quite condensed bracket characterisations and more straightforward geometric constructions.
First, consider the picture below.

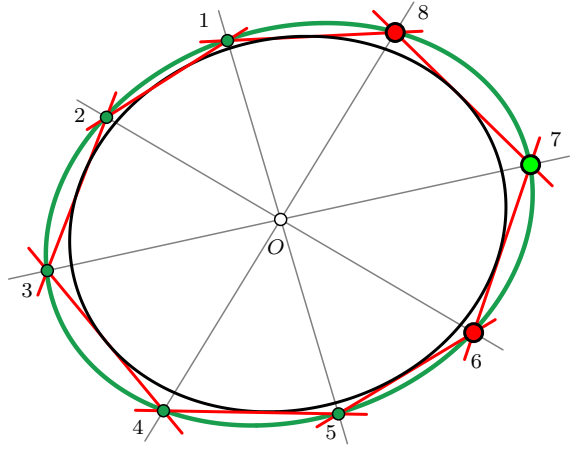
\begin{figure}[H]
\begin{center}
\includegraphics[width=0.65\textwidth]{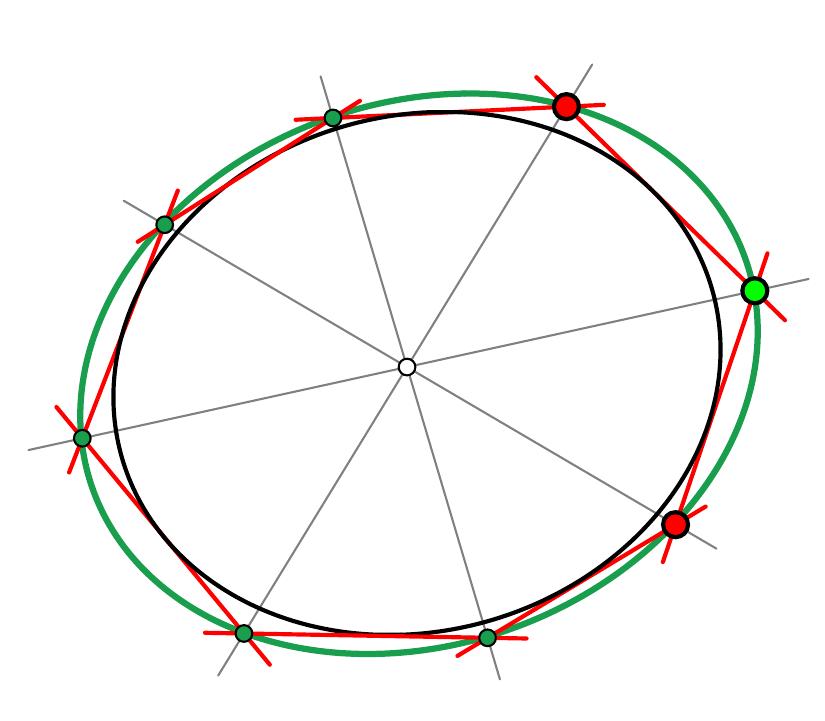}
\begin{picture}(0,0)
\put(-148,169){\footnotesize $1$}
\put(-195,133){\footnotesize $2$}
\put(-215,67){\footnotesize $3$}
\put(-173,15){\footnotesize $4$}
\put(-100,13){\footnotesize $5$}
\put(-45,40){\footnotesize $6$}
\put(-15,124){\footnotesize $7$}
\put(-65,172){\footnotesize $8$}
\put(-122,83){\footnotesize $O$}
\end{picture}
\end{center}
\captionof{figure}{A Poncelet octagon and a construction strategy.}
\end{figure}

It turns out that Poncelet $n$-gons with even $n$ posses a {\it center\/}: If one connects opposite points, the corresponding lines will all pass through one point $O$, which we call the center
of the polygon. This fact was already known to Darboux~\cite{Dar17}. Furthermore, the Poncelet octagon has a combinatorial symmetry where point 7 sits in ``the middle'' of points $1,\ldots, 5$. 
So, assume that we have constructed the position of the point 7 after the positions of $1,\ldots, 5$ were given. (We discuss how to do this in Figure \ref{fig:point7for8} and Theorem \ref{Point7BracketEq}.)
Then it is easy to construct points 6 and 8. 
Simply construct the center by intersecting $15$ and $37$. Then connect points 2 and 4 to the center, and intersect the corresponding lines with the opposite side of the conic to create 6 and 8.

In what follows we will employ that strategy, and strive for a construction and characterisation of point 7 from points $1,\ldots, 5$.

\begin{theorem} \label{Point7BracketEq}
Given a Poncelet octagon that has been transferred to $\mathbb{RP}^1$, the position of point 7 is characterised by the following bracket equation:
\[
[12][14][27][34][35][57] =[13][17][23][25][45][47].               
\]
\end{theorem}

\begin{proof}
Observe that in a Poncelet octagon we have 7 different Poncelet 7-chains of consecutive points.
By Theorem~\ref{thm:ProperPonceletChain}, each of these chains implies an algebraic condition with just two bracket monomials. 
The following equations arise in this way:
\[
\begin{array}{rcl}
 {\color{black}[14]} [24] [37] [56] &=& [15]{\color{red}[23]}[46]{\color{red} [47] }\\[1mm]
 {\color{red}[27]}  [28]  [15] {\color{red}[34]} &=& [37] [18] [24]  {\color{black}[25]}   \\[1mm]
 [48]  {\color{red}[12]}  {\color{red}[35]} [36] &=& [38] {\color{red} [13]} [26]  {\color{red}[45]}   \\[1mm]
 [38] [46] [15] [67] &=& [37]  {\color{black}[14]} [56] [68]   \\[1mm]
  {\color{red}[14]}  {\color{red}[57]} [26] [78] &=& [48]  {\color{red}[25]} [67]  {\color{red}[17]}   \\[1mm]
  {\color{black}[25]} [68] [37] [18] &=& [15] [36] [78] [28].   \\[1mm]
 \end{array}
\]
Multiplying all left sides and right sides, and cancelling determinants that occur on both sides yields
\begin{equation}\label{eq:octagonpt7}
[12][14][27][34][35][57] =[13][17][23][25][45][47].
\end{equation}
As a matter of fact, the red brackets are those that survive.
\end{proof}
\noindent

Next, we  want to create the position of point 7 from the positions of
$1\upto 5$, and prove the correctness of the construction.

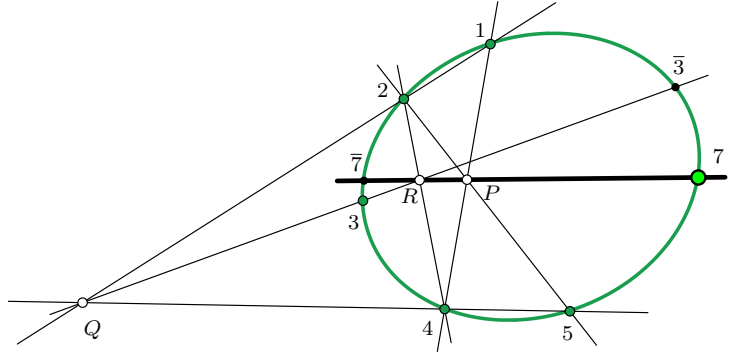
\begin{figure}[H]
\begin{center}
\includegraphics[width=.8\textwidth]{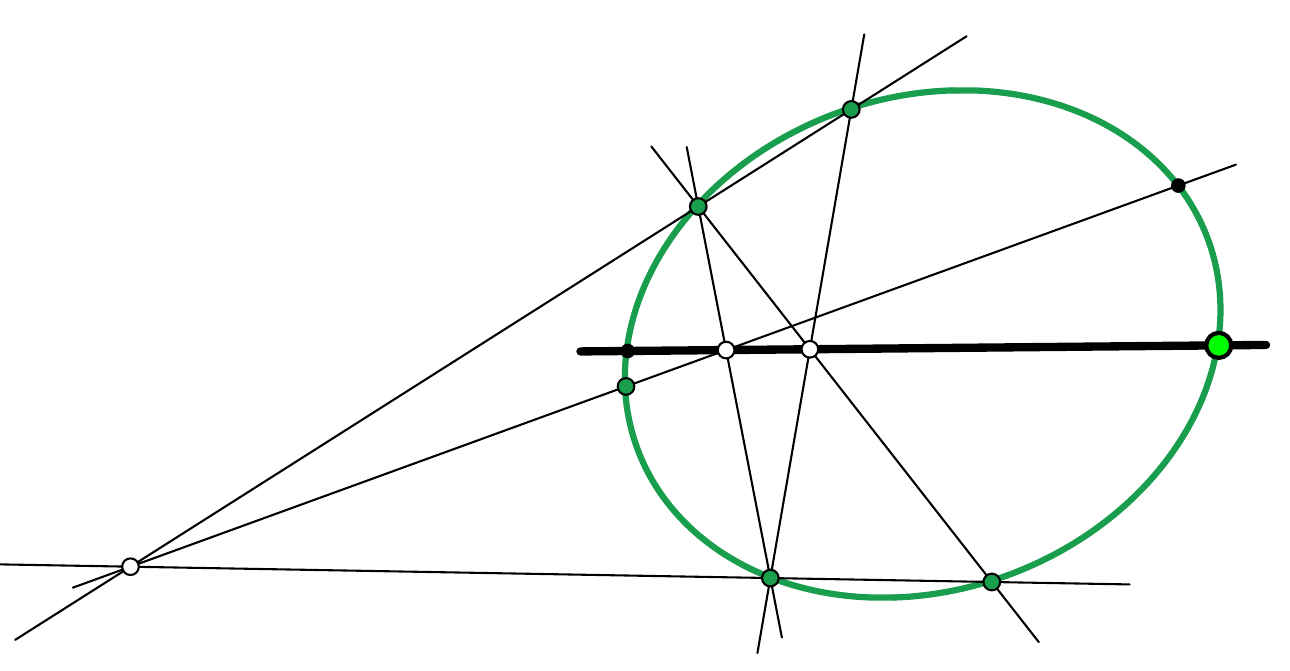}
\begin{picture}(0,0)
\put(-103,123){\footnotesize $1$}
\put(-140,100){\footnotesize $2$}
\put(-151,50){\footnotesize $3$}
\put(-28,109){\footnotesize $\overline{3}$}
\put(-123,10){\footnotesize $4$}
\put(-70,8){\footnotesize $5$}
\put(-13,75){\footnotesize $7$}
\put(-150,72){\footnotesize $\overline{7}$}
\put(-100,61){\footnotesize $P$}
\put(-131,60){\footnotesize $R$}
\put(-251,10){\footnotesize $Q$}
\end{picture}
\end{center}
\captionof{figure}{Construction of point 7 for the Poncelet 8-gon.}
\label{fig:point7for8}
\end{figure}

\begin{construction}
{\rm
The start of the construction is the five free points $1,\ldots, 5$.
Then we construct a few additional points and lines.
\begin{itemize}
\item[$P$:] intersection of $14$ and $25$,
\item[$Q$:] intersection of $12$ and $45$,
\item[$R$:] intersection of $24$ and $3Q$,
\item[$l$:] join of $R$ and $P$.
\end{itemize}

Finally, we intersect the line $l$ with the conic $C$ that passed through the initial 5 points.
This intersection will be called $7$, and it will be a suitable choice for constructing a Poncelet 8-gon.
}
\end{construction}
\medskip
\begin{theorem}
The construction above creates a valid point 7 for a Poncelet octagon.
\end{theorem}

\begin{proof}
In the construction above we again consider all incidences as described by quadset relations in $\mathbb{RP}^1$.
Points $\overline{3}$ and $\overline{7}$ (the other intersections of $3Q$ and $l$ with the conic, respectively) are not part of the construction, but will be used for the algebraic characterisation.
We have the quadset relations
\begin{itemize}
\item[$P$:] $(14\colon 25\colon 7\overline{7})$ are a quadset,\\[-6mm]
\item[$Q$:] $(12\colon 45\colon  3\overline{3})$ are a quadset,\\[-6mm]
\item[$R$:] $(24\colon 3\overline{3}\colon 7\overline{7})$ are a quadset,\\[-6mm]
\end{itemize}
\noindent
from which we can derive the equations:
\[
\begin{array}{rcl}
 [4\overline{7}] {\color{red}[21]} {\color{red}[75]} &=& {\color{red}[45]} [2\overline{7}] {\color{red}[71]}  \\[1mm]
 [2\overline{3}] {\color{red}[41]} {\color{red}[35]} &=& {\color{red}[25]} [4 \overline{3}] {\color{red}[31] }  \\[1mm]
 {\color{red}[27]} {\color{red}[34]} [\overline{7}\overline{3}] &=& [2\overline{3}] [37] [\overline{7}4]   \\[1mm]
 [4\overline{3}] [37] [\overline{7}2]  &=& {\color{red}[47]} {\color{red}[32]} [\overline{7}\overline{3}].   \\[1mm]
 \end{array}
\]
Multiplying left and right sides, and cancelling terms occurring on both sides
creates exactly the characterization
\[
[12][14][27][34][35][57] =[13][17][23][25][45][47] 
\]
 that was given in Theorem~\ref{Point7BracketEq}.
\end{proof}

\subsection{Doubling} \label{sect:doubling}

Next, we briefly sketch a method that allows us to geometrically construct a Poncelet $2n$-gon from a Poncelet $n$-gon. 
Together with our methods for geometrically constructing Poncelet $n$-gons for $3\leq n\leq 8$, this allows us to produce Poncelet $(2^k\cdot m)$-gons
for each  $3\leq m\leq 8$ and $k\geq 0$. However, the doubling construction is algebraically 
more heavy then the constructions we did so far (this is why we created the octagon separately).
The initial doubling requires  constructing the intersection of two conics. Algebraically, this operation cannot be done by only solving quadratic equations. 
It is algebraically as mighty as solving polynomials of degree 3. Thus, unlike the constructions for $n\leq8$, they cannot be performed by ruler and compass alone.
For the reader who wants to create similar constructions with a dynamic geometry program, we recommend one that is capable of dealing with intersections between 
two conics in a stable manner, like for instance Cinderella \cite{RiKo99}.

\begin{figure}[H]
\begin{center}
\includegraphics[width=.75\textwidth]{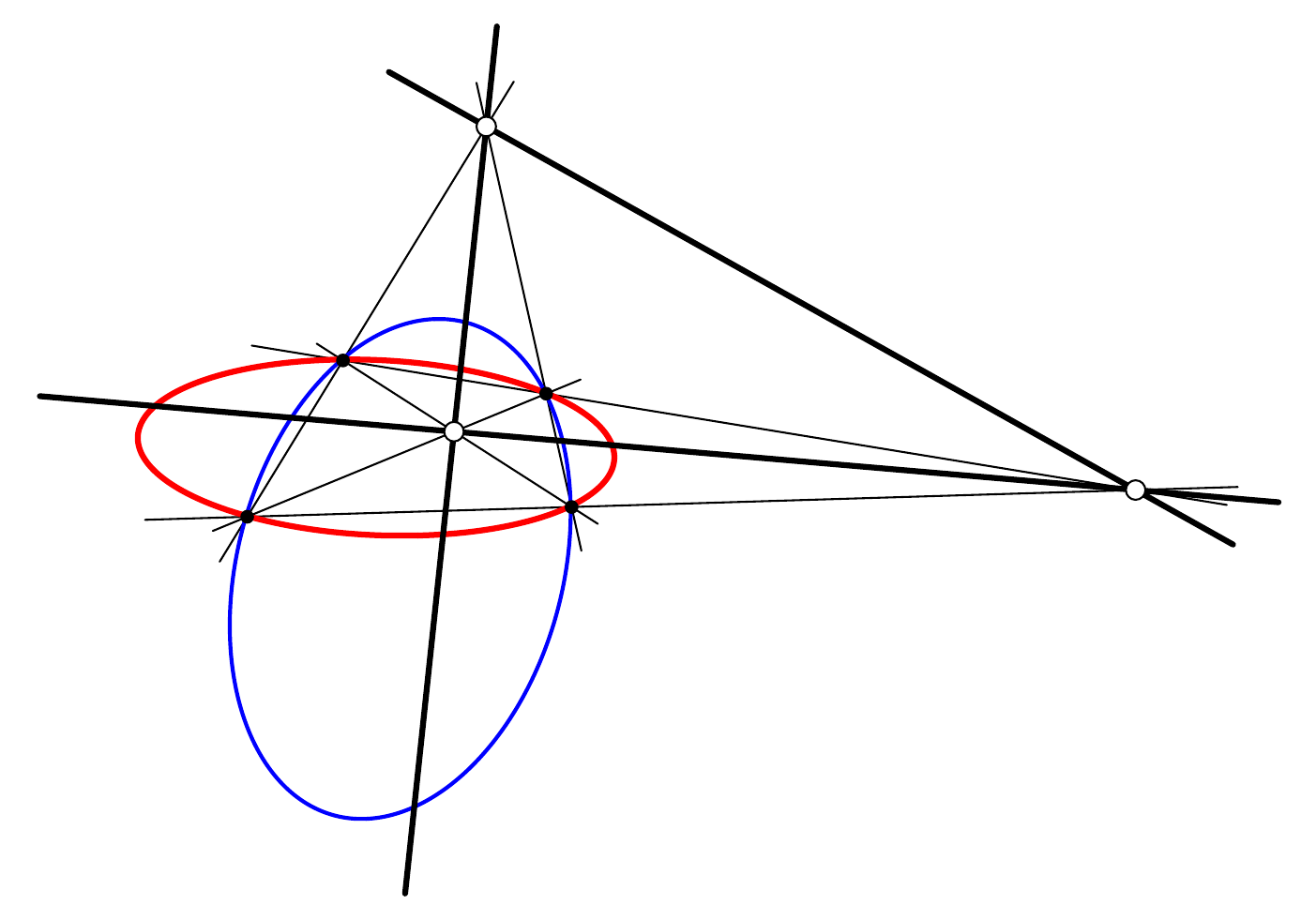}
\begin{picture}(0,0)
\put(-225,30){\footnotesize $\mathcal{A}$}
\put(-245,85){\footnotesize $\mathcal{B}$}
\put(-229,65){\footnotesize $r_1$}
\put(-149,105){\footnotesize $r_2$}
\put(-145,68){\footnotesize $r_3$}
\put(-202,114){\footnotesize $r_4$}
\put(-182,83){\footnotesize $O$}
\put(-44,70){\footnotesize $X$}
\put(-160,155){\footnotesize $Y$}
\put(-255,105){\footnotesize $x$}
\put(-180,3){\footnotesize $y$}

\end{picture}
\end{center}
\captionof{figure}{Constructing a local coordinate system.}\label{fig:rectify1}
\end{figure}
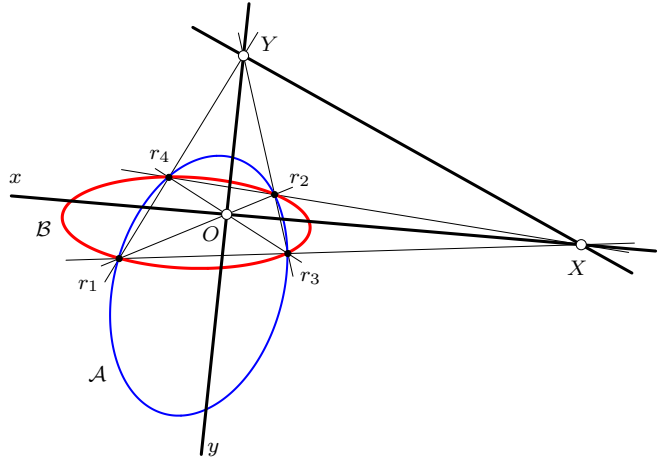

The construction for doubling that we will present without formal proof relies on the fact that by a suitable projective transformation the inner  (inscribed) conic of a Poncelet $n$-gon can  be mapped to the outer  (circumscribed) one. By choosing this transformation with care, after applying the transformation, the tangent points of the inner conic  are interleaved with the points on the outer conic  (see Figure \ref{fig:doub1}) forming a Poncelet $2n$-gon.
The following construction does the job:

\begin{construction}
Let $\mathcal{A}$ be the outer conic and $\mathcal{B}$ be the inner conic of a Poncelet polygon and assume that the points on $\mathcal{A}$ are in cyclic order $(p_1,\upto, p_n)$.
\begin{itemize}
\item[1.] Construct all four intersections $r_1\upto r_4$ of $\mathcal{A}$ and $\mathcal{B}$ (recall these may be complex).
\item[2.] Create the six lines $l_{ij}=r_i\vee r_j$.
\item[3.] Construct the three intersections 
$O=l_{12}\wedge l_{34}$,
$X=l_{13}\wedge l_{24}$,
$Y=l_{14}\wedge l_{23}$.
\item[4.] Construct the two lines 
$x=O\vee X$ and
$y=O\vee Y$.
\item[5.] Construct the two intersections
$a_{x1}$, $a_{x2}$ of $x$ with $\mathcal{A}$. Similarly, construct 
$a_{y1}$, $a_{y2}$ and $b_{x1}$, $b_{x2}$, $b_{y1}$, $b_{y2}$, the corresponding intersections with $\mathcal{B}$. 
\item[6.] Construct a projective transformation $\tau$ which is determined by
$\tau(b_{x1})= a_{x1}$,
$\tau(b_{x2})= a_{x2}$,
$\tau(b_{y1})= a_{y1}$ and
$\tau(b_{y2})= a_{y2}$.
\item[7.] Let $p_1\upto p_n$ be the points of the Poncelet polygon ordered cyclically,
and let $q_i$ be the touching point of $p_ip_{i+1}$ with $\mathcal{B}$; then
$p_1,\tau(q_1),p_2,\tau(q_2),\ldots$ forms a Poncelet $2n$-gon.

\end{itemize}
\end{construction}

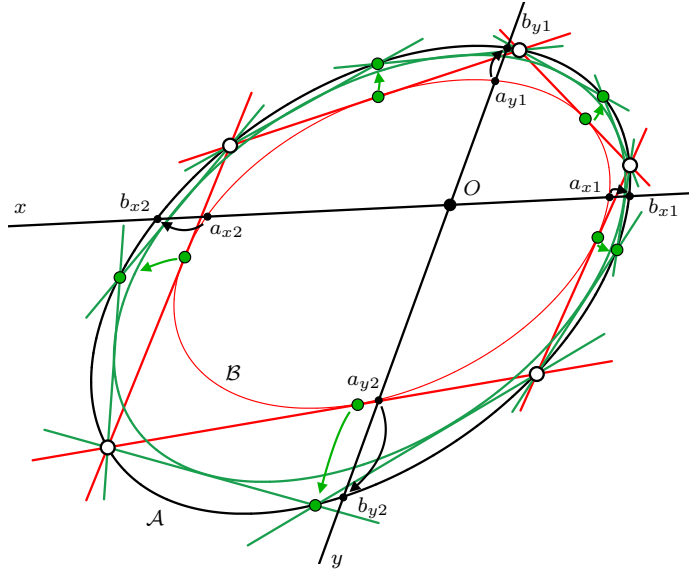
\begin{figure}[t]
\begin{center}
\includegraphics[width=.75\textwidth]{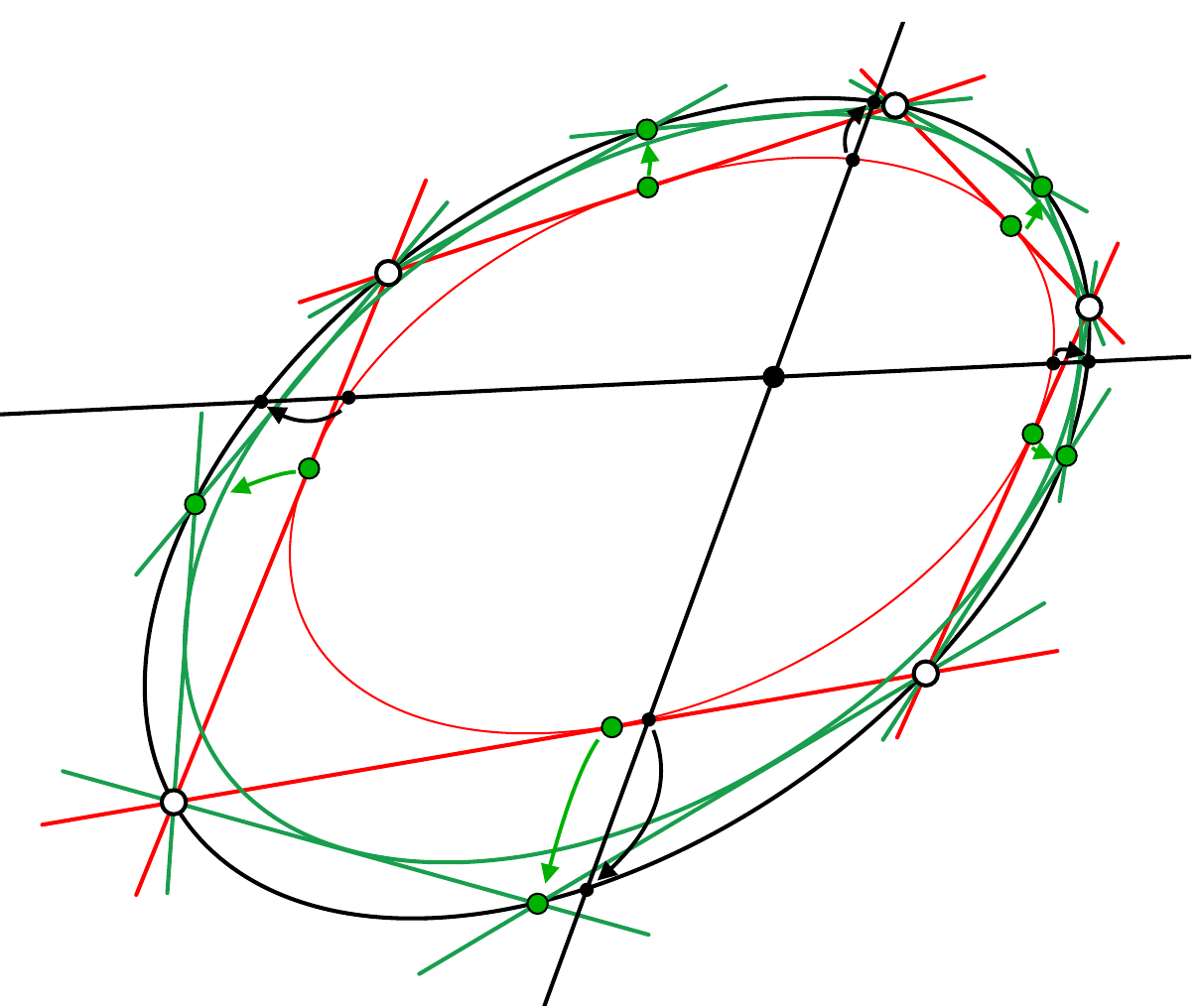}
\begin{picture}(0,0)
\put(-210,15){\footnotesize $\mathcal{A}$}
\put(-180,70){\footnotesize $\mathcal{B}$}
\put(-186,125){\footnotesize $a_{x2}$}
\put(-220,135){\footnotesize $b_{x2}$}
\put(-50,142){\footnotesize $a_{x1}$}
\put(-20,133){\footnotesize $b_{x1}$}
\put(-78,175){\footnotesize $a_{y1}$}
\put(-68,203){\footnotesize $b_{y1}$}
\put(-134,68){\footnotesize $a_{y2}$}
\put(-130,20){\footnotesize $b_{y2}$}
\put(-260,133){\footnotesize ${x}$}
\put(-140,0){\footnotesize ${y}$}
\put(-90,140){\footnotesize $O$}

\end{picture}
\end{center}
\captionof{figure}{Construction of a Poncelet 10-gon via doubling.
}\label{fig:doub1}
\end{figure}

This construction requires a few explanations. Firstly, there are several intrinsic ambiguities in the construction. Whenever we constructed intersections, we did not specify a specific order. 
This is intrinsic to the problem  and each choice leads to a proper result. 
Secondly, construction steps 1 to 4 reconstruct a kind of specific local coordinate system. The construction is shown in Figure \ref{fig:rectify1} for the case where the four intersections are real.
In general, we have to deal with the situation that some intermediate construction steps may produce complex elements. Algebraically this can be overcome easily by performing all construction 
steps in $\mathbb{CP}^2$. The dynamic geometry system Cinderella~\cite{RiKo99} fully supports complex geometric elements, so it is possible to perform these constructions with relative ease.
Thirdly, the construction requires the explicit construction of a projective transformation $\tau$ and the mapping of certain elements with respect to this transformation $\tau$. In principle, this can be reduced 
to {\it join} and {\it meet} operations, but it is complicated in practice.

Figure~\ref{fig:doub1} shows the construction for $n=10$. From a Poncelet 5-gon
(black conic and red conic),
first the local coordinate system is constructed (the construction here is all complex and not shown in the image).
It  results in the point $O$ and the lines $x$ and $y$.
 Although the intersections of the red and the black conic are complex, the points $O$ and the two axes $x$ and $y$ turn out to be real. 
Their intersection with the conics defines the projective transformation $\tau$ represented by the arrows. 
 Mapping the touching points to the conic defines the projective transformation with which the touching points (green) are mapped. 
 This results in a collection of 10 points on the conic $\mathcal{A}$ that form a Poncelet $10$-gon.

\subsection{Poncelet 9-gons}

Considering the list at the end of Section~\ref{sectalgebra}
shows that the algebraic degree necessary to construct a Poncelet 9-gon
turns out to be ``just'' three. This hints at the possibility of constructing such a polygon with the aid of intersecting two conics, an operation that algebraically corresponds to solving cubic (or equivalently quartic) equations.
Although the construction is comparably simple, proving its correctness is by far more involved. The construction and its relatives will be discussed in a companion paper \cite{RGS24}.

Assume that the 5 points  $1,2,3,5,6$ of an initial Poncelet chain $1,2,3,4,5,6$ on a conic $\mathcal{A}$ are given.  Since the construction of a Poncelet 9-gon leads to a cubic equation (compare Section~\ref{sectalgebra}), there will be in general three positions for point 4 that makes the chain the beginning of a Poncelet 9-gon. It is reasonable to assume that these three positions can be constructed as the intersection of two conics. Two conics have four intersections. So one is aiming for a conic  $\mathcal{C}$ that intersects  $\mathcal{A}$ in four points, one of which is already part of the construction and where the remaining
three intersections are the possible  positions of the remaining point. This is analogous to the algebraic situation in which a quartic equation reduces to a cubic equation if one solution is already known.

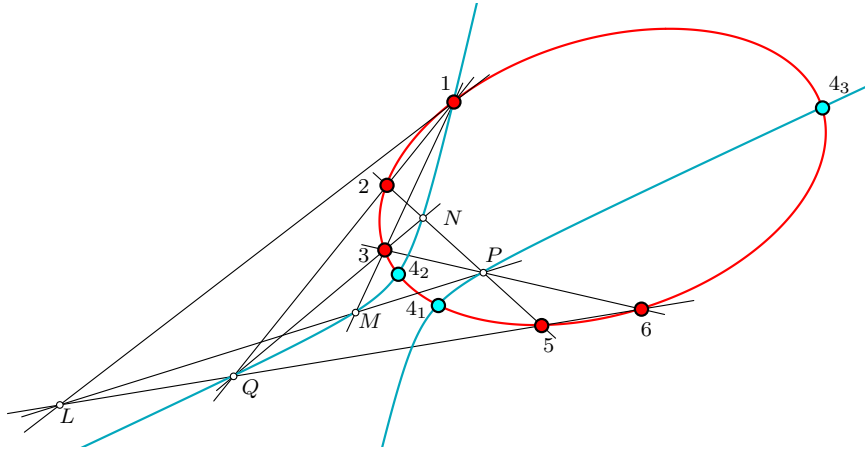
\begin{figure}[t]
\begin{center}
\includegraphics[width=.95\textwidth]{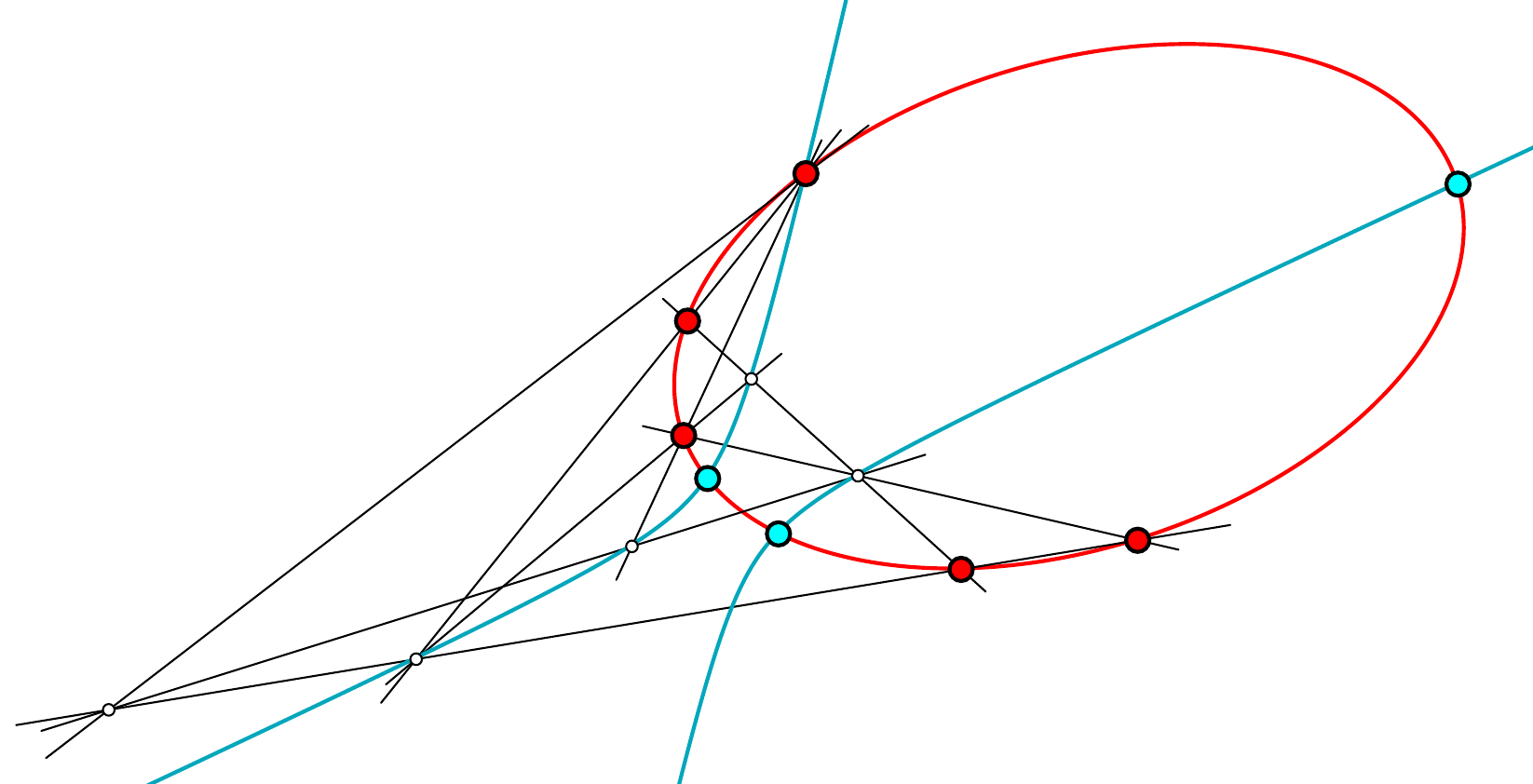}
\begin{picture}(0,0)
\put(-164,135){\footnotesize $1$}
\put(-195,96){\footnotesize $2$}
\put(-195,70){\footnotesize $3$}
\put(-176,65){\footnotesize $4_2$}
\put(-177,50){\footnotesize $4_1$}
\put(-125,36){\footnotesize $5$}
\put(-88,42){\footnotesize $6$}
\put(-17,135){\footnotesize $4_3$}
\put(-147,70){\footnotesize $P$}
\put(-163,84){\footnotesize $N$}
\put(-308,9){\footnotesize $L$}
\put(-195,45){\footnotesize $M$}
\put(-239,20){\footnotesize $Q$}

\end{picture}
\end{center}
\captionof{figure}{Construction of possible points for a Poncelet 9-gon. }\label{fig:9gon1}
\end{figure}

The following construction is the result of an exhaustive 
search among promising construction sequences that has been performed in
\cite{RGS24}. To the best of our knowledge it is the shortest construction sequence that creates a Poncelet 9-gon from 5 points in general position.
It starts with the points $1,2,3,5,6$ and constructs three possible positions $4_1,4_2,4_3$ such that for each $i=1,2,3$ the chain 
$1,2,3,4_i,5,6$ is the initial sequence of a Poncelet 9-gon.
Finding the construction is a bit intricate, and we here present the construction without a proof. For details see \cite{RGS24}.

\begin{construction}
We start with five points $1,2,3,5,6$ such that no triple is collinear in the projective plane.

\begin{itemize}
\item[$\mathcal{A}$:] conic through $1,2,3,5,6$,
\item[$P$:] intersection of $25$ and $36$,
\item[$Q$:] intersection of $12$ and $56$,
\item[$L$:] intersection of $56$ and $\mathit{tangent}(1,\mathcal{A})$,
\item[$M$:] intersection of $13$ and $LP$,
\item[$N$:] intersection of $24$ and $Q3$,
\item[$\mathcal{C}$:] conic through $1,P,Q,N,M$,
\item[$4_i$:] the three intersections of $\mathcal{A}$ and $\mathcal{C}$ other than $1$.
\end{itemize}
\end{construction}

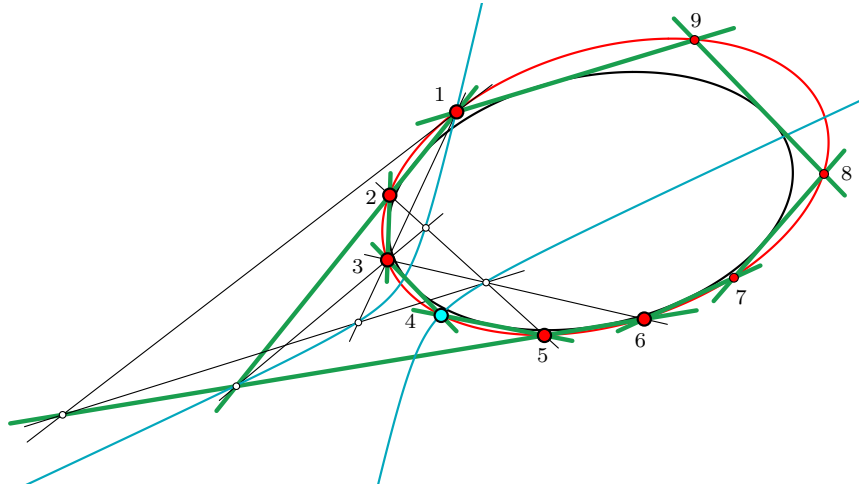
\begin{figure}[t]
\begin{center}
\includegraphics[width=.95\textwidth]{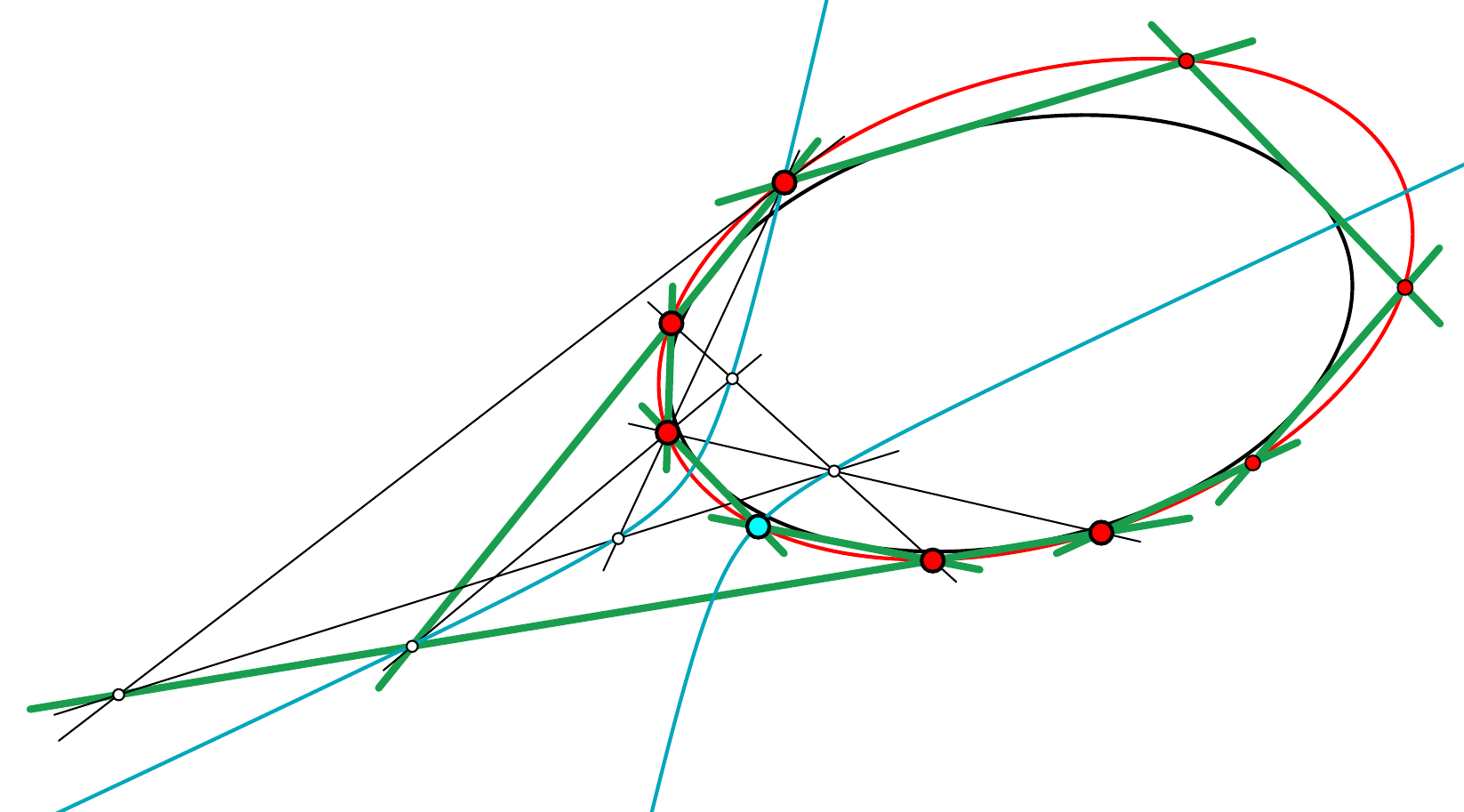}
\begin{picture}(0,0)
\put(-164,145){\footnotesize $1$}
\put(-190,106){\footnotesize $2$}
\put(-195,80){\footnotesize $3$}
\put(-175,59){\footnotesize $4$}
\put(-125,46){\footnotesize $5$}
\put(-88,52){\footnotesize $6$}
\put(-50,68){\footnotesize $7$}
\put(-10,115){\footnotesize $8$}
\put(-67,173){\footnotesize $9$}

\end{picture}
\end{center}
\captionof{figure}{Finishing the Poncelet 9-gon for one of the constructed points.}\label{fig:9gon2}
\end{figure}

Each of the constructed intersections $4_i$ can be used as an admissible choice for concluding the chain to a Poncelet 9-gon. Figure \ref{fig:9gon2} 
shows the Poncelet 9-gon for one of the constructed points.

Similar to the cases of the 7-gon and 8-gon we also here searched for a short bracket polynomial that characterises the geometric situation. The shortest one that we could find charaterises when the sequence
$1,2,3,4,5,7$ is part of a Poncelet 9-gon $1,2,3,4,5,6,7,8,9$.
It can be given as  a rank 1 condition for the points on the conic in the same spirit we used for the other characterizations
\[
\begin{array}{l@{}lr}
&  [12][14][15][27][27][34][35][35][47]\\
 -& [15][17][17][23][23][24][34][45][57]\\
 -& [12][14][17][24][25][35][37][37][45]&=\quad 0
\end{array}
\]
Again we refer to \cite{RGS24} for details, proofs and methods how to find
the constructions and bracket polynomials.

\subsection{Long Poncelet chains}

Our characterisation of creating a seventh point in a Poncelet chain after the first six are given can easily be used to
create an arbitrarily long Poncelet chain from 6 initial points. The equation 
\[[74][16][54][32]=[72][14][56][34]\]
is linear in the seventh point. Hence after $1\upto 6$ are given, its unique coordinates in $\mathbb{RP}^1$ can be calculated by
\[7=[16][54][32]\cdot 4+[14][56][34]\cdot 2.\] 
This can be easily checked by plugging this representation in the 7-chain equation.

\medskip
Nevertheless, it is also desirable to have a concrete method for geometrically constructing such a point from the initial $6$ points
on a conic. Since the above equation  only contains the elementary arithmetic operations, it can be expected that
the point can be constructed exclusively by {\it join} and {\it meet} operations. In what follows we  describe such a construction, 
prove its correctness and show how it can be used  to create arbitrarily long chains in a nice iterative way.

\medskip
Before we describe the construction, we introduce an incidence theorem that will be useful in that context: The Conic Butterfly Theorem, that has recently been proved by \cite{Iz15}. 
We give an independent bracket-theoretic proof of the theorem.

\begin{theorem}
Let $A_1,A_2,A_3,A_4$ be four distinct points on a conic $\mathcal{C}$  and let $l$ be an arbitrary line. Let the lines $A_iA_{i+1}$
intersect $l$ in points $X_i$ (with indices computed modulo $4$). If $B_1\upto B_4$ are four points on  $\mathcal{C}$  such that for 
$i=1\upto 3$ the line $B_iB_{i+1}$ passes through $X_i$, then $B_4B_1$ passes through $X_4$.
\end{theorem}

\begin{proof}
Refer to the Figure~\ref{fig:butterfly} for the geometric situation. We introduce two additional points $P$ and $Q$ that are the intersections of $l$ with the conic. 
(They may be complex points if $l$ and $\mathcal{C}$  do not intersect, but they occur only in our calculations that work over general fields.) 
We again represent the points on the conic $A_1\upto A_4, B_1\upto B_4,P,Q$ by their $\mathbb{RP}^1$ coordinates.

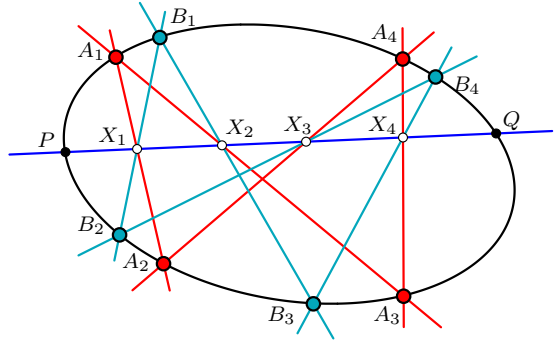
\begin{figure}[H]
\begin{center}
\includegraphics[width=.65\textwidth]{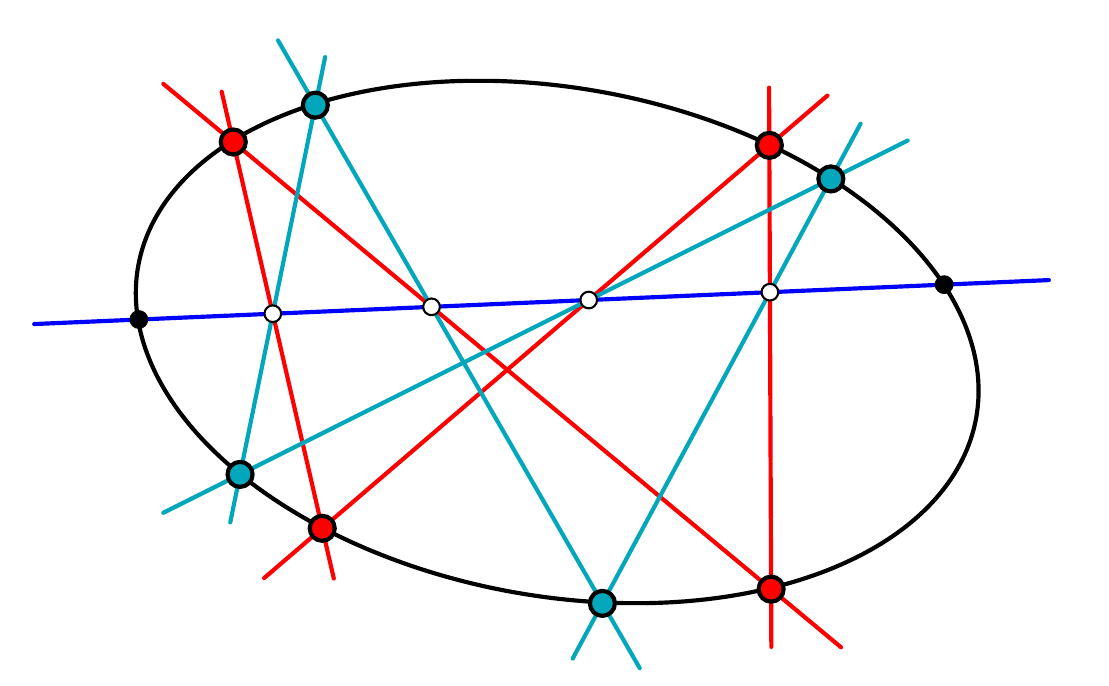}
\begin{picture}(0,0)
\put(-195,110){\footnotesize $A_1$}
\put(-178,30){\footnotesize $A_2$}
\put(-83,10){\footnotesize $A_3$}
\put(-84,117){\footnotesize $A_4$}
\put(-160,123){\footnotesize $B_1$}
\put(-195,44){\footnotesize $B_2$}
\put(-123,10){\footnotesize $B_3$}
\put(-53,97){\footnotesize $B_4$}
\put(-187,78){\footnotesize $X_1$}
\put(-139,81){\footnotesize $X_2$}
\put(-117,82){\footnotesize $X_3$}
\put(-85,82){\footnotesize $X_4$}
\put(-34,84){\footnotesize $Q$}
\put(-210,76){\footnotesize $P$}
\end{picture}
\end{center}
\captionof{figure}{The Conic Butterfly Theorem. 
}\label{fig:butterfly}
\end{figure}

At each inner point $X_i$ three lines meet. In the $\mathbb{RP}^1$ setup this  translates again to relations.
We have the following four quadsets:
\[
\begin{array}{c}
(A_1A_2\colon B_2B_1\colon Q P);\quad
(A_2A_3\colon B_3B_2\colon P Q);\\ 
(A_3A_4\colon B_4B_3\colon Q P);\quad
(A_4A_1\colon B_1B_4\colon P Q).\\
\end{array} 
\]
They can be expressed by the following four equations: 
 \[
\begin{array}{c}
[A_1B_1][QA_2][B_2P] = [B_2A_2][A_1P][QB_1] \\[1mm]
[A_2B_2][PA_3][B_3Q] = [B_3A_3][A_2Q][PB_2] \\[1mm]
[A_3B_3][QA_4][B_4P] = [B_4A_4][A_3P][QB_3] \\[1mm]
[A_4B_4][PA_1][B_1Q] = [B_1A_1][A_4Q][PB_4] \\[1mm]
\end{array} 
\]
Applying the usual multiply-and-cancel trick shows that each quadset relation can be concluded from the other three.
\end{proof}

\bigskip

Now, we make use of this result to constructively create a seventh point in a Poncelet chain from the previous ones.
The basic construction is captured by Figure~\ref{fig:chain3}.
We start with 6 points on a conic, and perform the following construction steps
(the slightly non-standard labelling will be explained later).

\begin{construction}
{\rm
Let $1,\ldots, 6$ be six distinct points on a conic $\mathcal{A}$. Construct
\begin{itemize}
\item[1.] $B_3=(1\vee 2)\wedge(3\vee 4)$
\item[2.] $G_3=(1\vee 4)\wedge(3\vee 6)$
\item[3.] $G_4=(G_3\vee B_3)\wedge(2\vee 5)$
\item[4.] $B_6=(G_3\vee B_3)\wedge(4\vee 5)$
\item[5.] $7=(G_4\vee 4)\wedge(B_6\vee 6)$
\end{itemize}
}
\end{construction}

\begin{theorem}
The last construction produces a seventh point in a Poncelet  chain.
\end{theorem}

\begin{proof}
Comparing the construction with the Conic Butterfly Theorem shows that the construction is almost identical.
The only difference is that one point plays a double role since it is used by both quadrilaterals of the Butterfly Theorem.
The role of the pivoting points is played by $G_3,G_4,B_3,B_6$, and the two quadrilateral chains  are
\[
\begin{array}{l}
7\xrightarrow{G_4}4\xrightarrow{B_3}3\xrightarrow{G_3}6\xrightarrow{B_6}7\\
5\xrightarrow{G_4}2\xrightarrow{B_3}1\xrightarrow{G_3}4\xrightarrow{B_6}5.
\end{array}\]
The letters above the arrows indicate the corresponding pivot point.
The Conic Butterfly Theorem implies that the constructed point 7 is also on the conic $\mathcal{C}$.

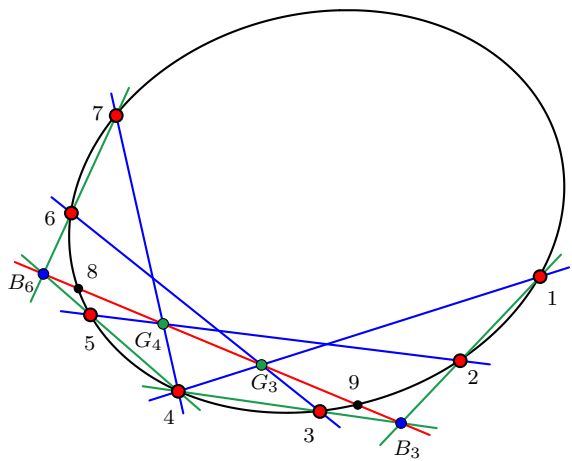
\begin{figure}[H]
\begin{center}
\includegraphics[width=.8\textwidth]{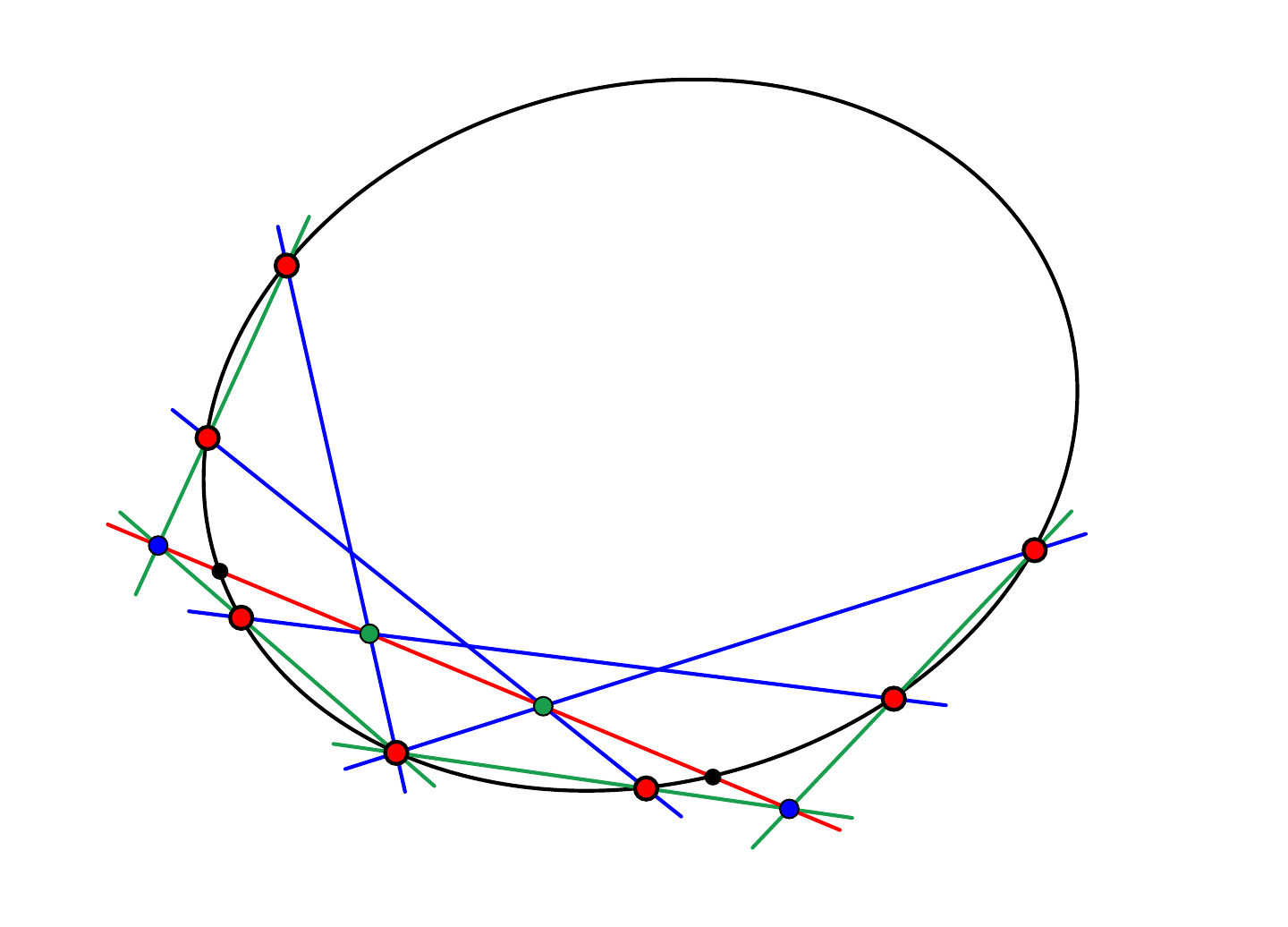}
\begin{picture}(0,0)
\put(-55,70){\footnotesize $1$}
\put(-85,40){\footnotesize $2$}
\put(-147,18){\footnotesize $3$}
\put(-200,25){\footnotesize $4$}
\put(-230,53){\footnotesize $5$}
\put(-245,100){\footnotesize $6$}
\put(-227,140){\footnotesize $7$}
\put(-113,13){\footnotesize $B_3$}
\put(-167,38){\footnotesize $G_3$}
\put(-211,54){\footnotesize $G_4$}
\put(-259,76){\footnotesize $B_6$}
\put(-229,80){\footnotesize $8$}
\put(-130,38){\footnotesize $9$}
\end{picture}
\end{center}
\kern-3mm
\captionof{figure}{Construction point 7 from $1\upto 6$.}\label{fig:chain3}
\end{figure}

\medskip
Since all points $1\upto 7$ lie on the conic, we can again represent them by their $\mathbb{RP}^1$ coordinates.
We introduce two additional points $8$ and $9$ that are the intersection of the red line and the conic. It only remains to show that the points satisfy the 7-chain condition. 
Again, the pivot points encode quadset conditions, and they in turn create the following four determinant conditions:
\[
\begin{array}{c}
{\color{red}[23]}[49][81]=[29]{\color{red}[31]}[83]\phantom{.}\\[1mm]
[49]{\color{red}[61]}[83]={\color{red}[43]}[69][81]\phantom{.}\\[1mm]
{\color{red}[47]}[69][85]=[49]{\color{red}[65]}[87]\phantom{.}\\[1mm]
[29]{\color{red}[45]}[87]={\color{red}[27]}[49][85].\\[1mm]  
\end{array} 
\]
Those that survive after the usual cancellation are marked red. They create exactly the  7-chain condition:
\[[74][16][54][32]=[72][14][56][34].\]
This is exactly what was desired.
\end{proof}

The advantage of this construction is that it can easily be iterated to create Poncelet chains of arbitrary length.
You might have observed that there are two ``missing" blue points $B_4=(2\vee 3)\wedge(4\vee 5)$ and $B_5=(4\vee 5)\wedge(5\vee 6)$. 
After adding them into the construction one can proceed iteratively by creating three more points and three more lines in each iteration step.

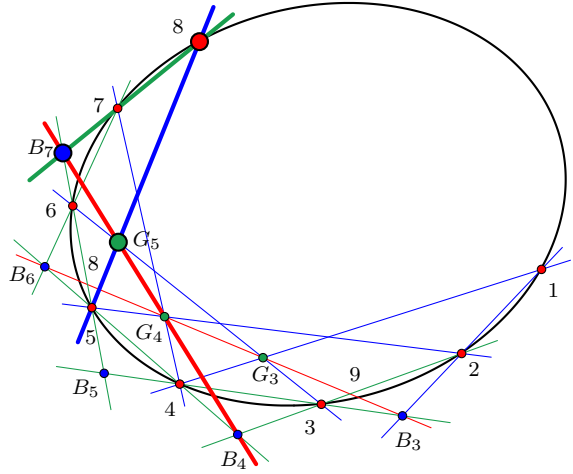
\begin{figure}[h]
\begin{center}
\includegraphics[width=.8\textwidth]{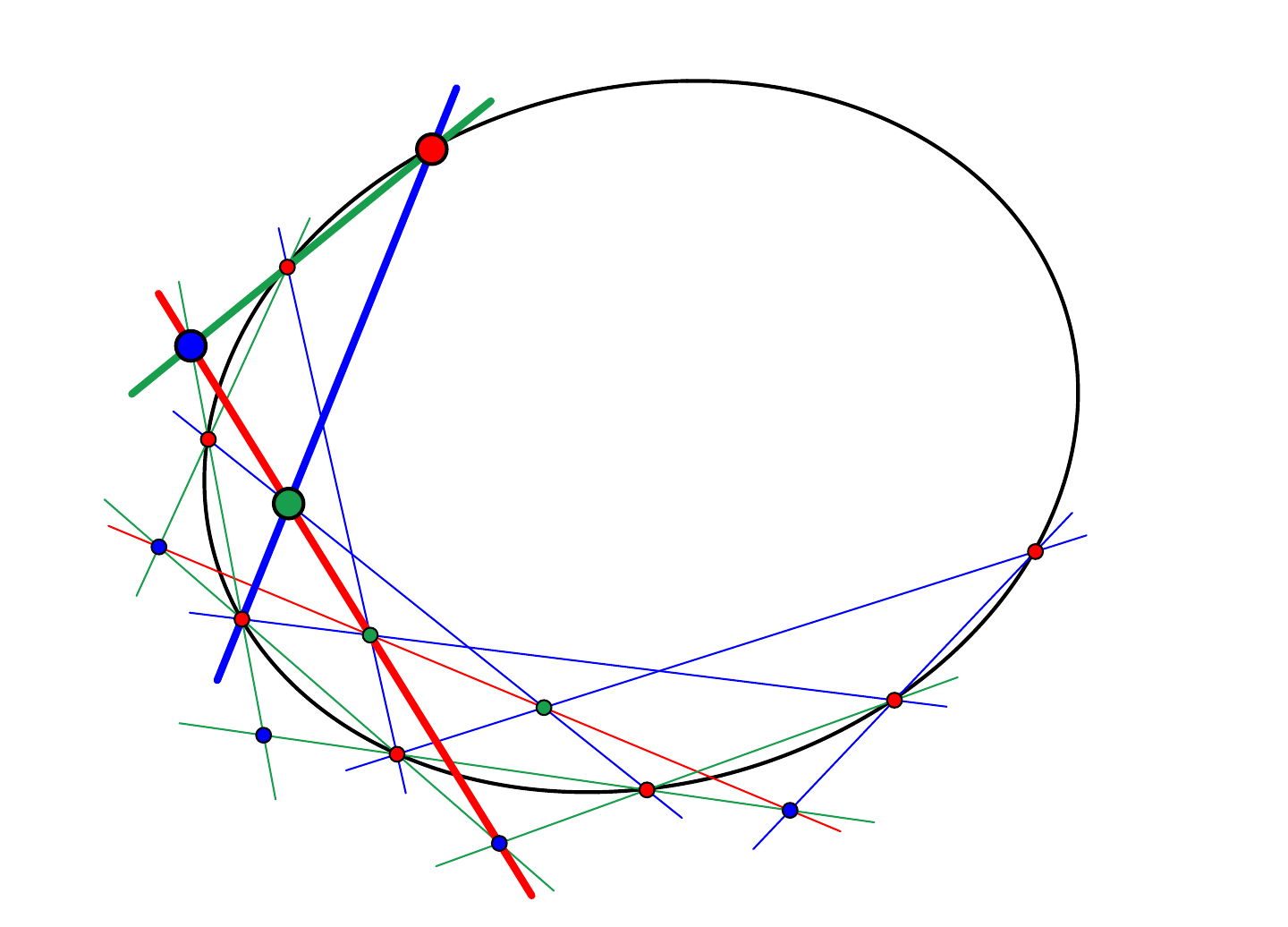}
\begin{picture}(0,0)
\put(-55,70){\footnotesize $1$}
\put(-85,40){\footnotesize $2$}
\put(-147,18){\footnotesize $3$}
\put(-200,25){\footnotesize $4$}
\put(-230,52){\footnotesize $5$}
\put(-245,100){\footnotesize $6$}
\put(-227,140){\footnotesize $7$}
\put(-113,13){\footnotesize $B_3$}
\put(-167,38){\footnotesize $G_3$}
\put(-211,54){\footnotesize $G_4$}
\put(-212,89){\footnotesize $G_5$}
\put(-259,76){\footnotesize $B_6$}
\put(-179,6){\footnotesize $B_4$}
\put(-235,33){\footnotesize $B_5$}
\put(-252,124){\footnotesize $B_7$}
\put(-229,80){\footnotesize $8$}
\put(-130,38){\footnotesize $9$}
\put(-197,170){\footnotesize $8$}
\end{picture}
\end{center}
\kern-3mm
\captionof{figure}{The iterative construction step.}
\label{fig:iterating}
\end{figure}

\begin{construction}
{\rm
Let $1\upto i$, $G_4\upto G_{i-2}$ and $B_3\upto B_{i-1}$ be constructed either by the initialisation procedure or by applying this step iteratively. 
Then we can construct $i+1$, $G_{i-2}$ and $B_i$ by
\begin{itemize}
\item[3.] $G_{i-2}=(G_{i-4}\vee B_{i-1})\wedge((i-4)\vee (i-2)$
\item[4.] $B_i=(G_{i-3}\vee B_{i-1})\wedge((i-2)\vee (i-1))$
\item[5.] $(i+1)=(G_{i-2}\vee (i-2)))\wedge(B_{i}\vee i)$
\end{itemize}
}
\end{construction}

With the color coding of our pictures in each such step a multicolored
triangle consisting of a red/green/blue point and a red/green/blue line is added.
Figure~\ref{fig:iterating} shows the construction of point 8 in the Poncelet chain, emphasising this triangle.

The procedure may be  continued for an arbitrarily long time. It may close up after $n$ steps if the initial sequence belongs to a
Poncelet $n$-gon. In this case, the construction produces a $(3n_{4})$ configuration. But it may equally well go on forever, either cycling around the conic or bouncing back and forth. 
Figure \ref{fig:cycling} shows several construction steps of a cycling situation.

\begin{figure}[H]
\begin{center}
\includegraphics[width=.9\textwidth]{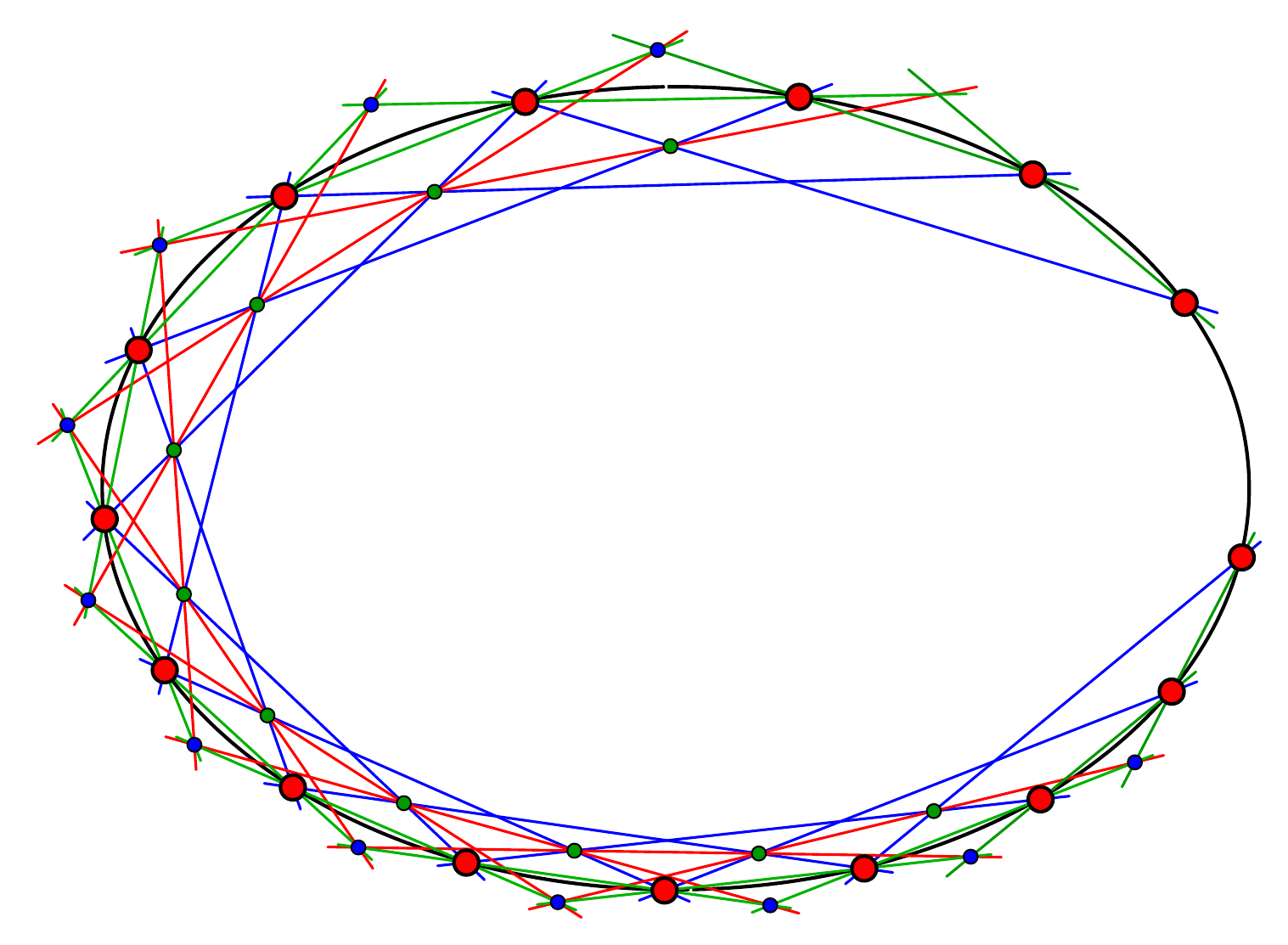}
\begin{picture}(0,0)
\end{picture}
\end{center}
\kern-3mm
\captionof{figure}{Iterating the construction.}
\label{fig:cycling}
\end{figure}

 Here are a few observations with respect to this construction
\begin{itemize}
\item Each line in the inner part of the construction contains exactly 4 points of the construction. Lines of one color contain only points of the other two colors.
\item Each point in the inner part of the construction is incident to exactly 4 lines of the construction. The lines through each point are exactly the other two colors.
\item We constructed three rings of points: red, green and blue. Not only do the red points
lie on a conic, but the blue and the green points also each lie on their own conics. One might see this by observing that from an incidence point of view the role of red/green/blue is symmetric.
\item Moreover, the blue and the green rings of points each form Poncelet chains.
\item All lines of one color are tangent to a common conic. 
\end{itemize}

Let us close with an interesting numerical experiment. We started our construction with 6 points on a conic, and then proceeded exclusively with join and meet operations. 
We could do the same join and meet construction for arbitrary initial points.
Figure \ref{fig:infty} shows what happens if we disturb the position point 6 slightly away from its place on the conic.
If one continues the construction in the other direction as well, one might consider this as a $(\infty_4)$ configuration.

\begin{figure}[H]
\begin{center}
\includegraphics[width=1\textwidth]{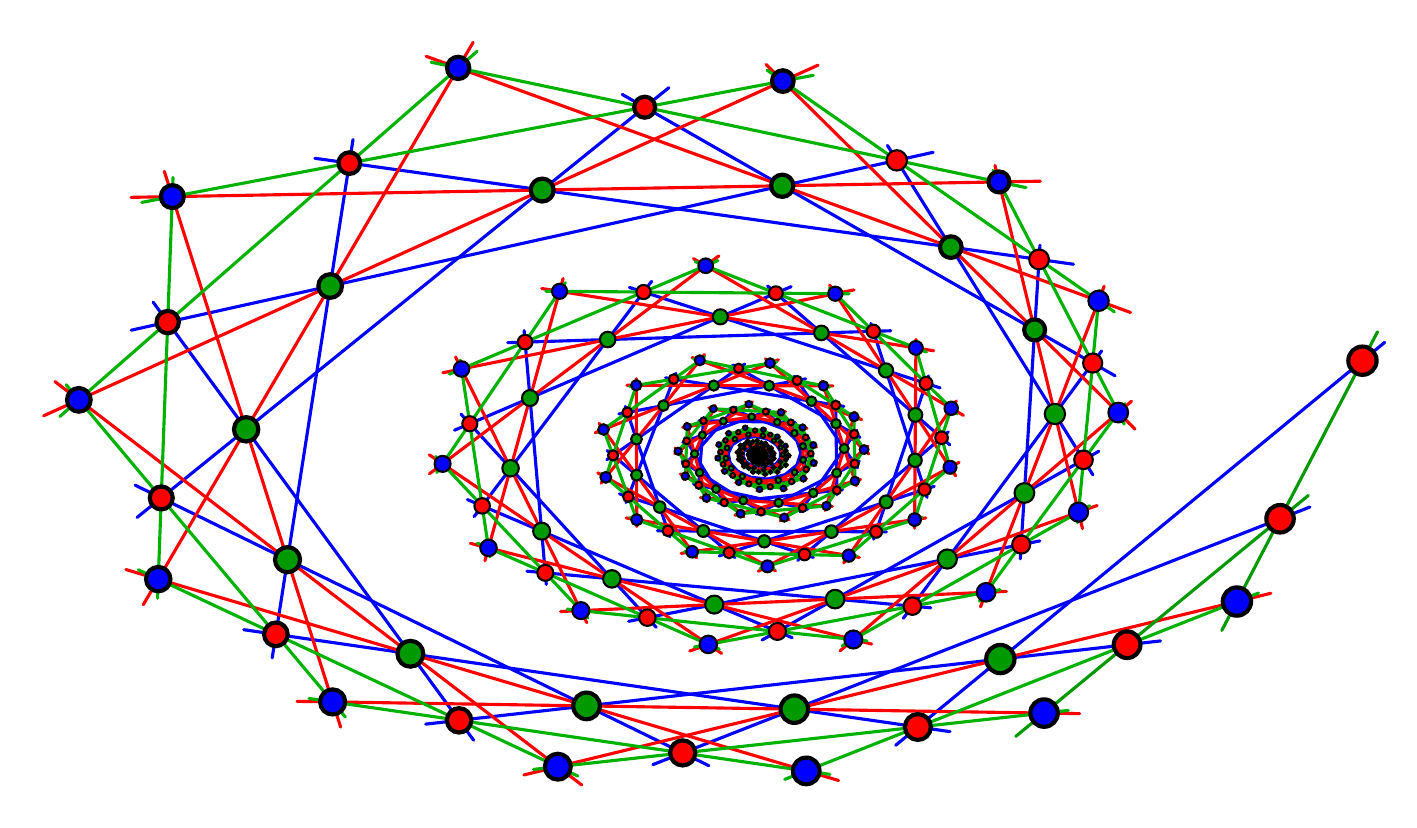}
\begin{picture}(0,0)
\end{picture}
\end{center}
\kern-3mm
\captionof{figure}{Iterating with other initial points.}
\label{fig:infty}
\end{figure}

\paragraph{\bf Acknowledgements.}

 We are grateful to A. Akopyan, Tim Reinhardt and Lena Polke for a useful discussion.
ST was supported by NSF grants DMS-2005444 and DMS-2404535, and by a Mercator fellowship within the CRC/TRR 191. 
GG
was supported by the
Hungarian National Research,
Development and Innovation Office,
OTKA Grant No. SNN 132625.

\bibliographystyle{plain} 
{\small
\bibliography{refsd}

\begin{thebibliography}{10}

\bibitem{BGRGT24a}
Leah~Wrenn Berman, G{\'a}bor G{\'e}vay, Serge Tabachnikov, and J{\"u}rgen
  Richter-Gebert.
\newblock When {G}r{\"u}nbaum meets {P}oncelet -- infinite classes of movable
  $(n_4)$ configurations.
\newblock Technical report, 2024.
\newblock in preparation.

\bibitem{RGS24}
Leah~Wrenn Berman, J{\"u}rgen Richter-Gebert, and Benjamin S{\"o}lch.
\newblock Constructing a poncelet 9-gon.
\newblock in preparation, 2024.

\bibitem{Cox92}
Harold Scott~Macdonald Coxeter.
\newblock {\em The Real Projective Plane}.
\newblock Springer, New York, 1992.

\bibitem{Cox94}
Harold Scott~Macdonald Coxeter.
\newblock {\em Projective Geometry}.
\newblock Springer, New York, Berlin, 1994.

\bibitem{Dar17}
Gaspard Darboux.
\newblock {\em Principes de Geometrie Analytique}.
\newblock Paris, Gauther-Villars, 1917.

\bibitem{GrRi90}
Branko Gr{\"u}nbaum and John~F. Rigby.
\newblock The real configuration ($21_4$).
\newblock {\em J. Lond. Math. Soc.}, 41:336--346, 1990.

\bibitem{Iz15}
Ivan Izmestiev.
\newblock A porism for cyclic quadrilaterals, butterfly theorems, and
  hyperbolic geometry.
\newblock {\em Amer. Math. Monthly,}, 122(5):467--475, 2015.

\bibitem{RG95}
J{\"u}rgen Richter-Gebert.
\newblock Mechanical theorem proving in projective geometry.
\newblock {\em Ann. Math. Artif. Intell.}, 13(1-2):139--172, 1995.

\bibitem{RG11}
J{\"u}rgen Richter-Gebert.
\newblock {\em Perspectives on {P}rojective {G}eometry}.
\newblock Springer, 2011.

\bibitem{RiKo99}
J{\"u}rgen Richter-Gebert and Ulrich Kortenkamp.
\newblock The {I}nteractive {G}eometry {S}oftware {C}inderella.
\newblock {\em Springer-Verlag, Berlin}, 1999.

\end{thebibliography}
}
\medskip

\noindent
{\footnotesize
{\sc Department of Mathematics \& Statistics, University of Alaska Fairbanks, USA}\\
Email address: {\tt lwberman@alaska.edu}

\noindent
{\sc Bolyai Institute, University of Szeged, Hungary}\\
Email address: {\tt gevay@math.u-szeged.hu}

\noindent
{\sc Department of Mathematics, Technical University of Munich, Germany}\\
Email address: {\tt richter@tum.de}

\noindent
{\sc Department of Mathematics, Penn State University, USA}\\
Email address: {\tt tabachni@math.psu.edu}

}
\end{document}